\documentclass[10pt,letterpaper,leqno,pdftex]{amsart}
\usepackage[all]{xy}
\usepackage[mathscr]{euscript}
\usepackage{amsmath, amsfonts, amssymb, mathrsfs, stmaryrd}
\usepackage{appendix}
\usepackage{amsthm}
\usepackage{comment}
\usepackage{bm}
\usepackage{lipsum}
\newcommand{\sm}{\smallsetminus}

\newcommand{\Oo}{\mathcal{O}}

\newcommand{\AAa}{\mathscr{A}}

\newcommand{\CC}{\mathbb{C}}
\newcommand{\RR}{\mathbb{R}}

\newcommand{\ZZ}{\mathbb{Z}}

\newcommand{\ZZZ}{\mathscr{Z}}
\newcommand{\DDD}{\mathscr{D}}

\newcommand{\MMM}{\mathscr{M}}

\newcommand{\OOO}{\mathscr{O}}

\newcommand{\VVV}{\mathscr{V}}
\newcommand{\III}{\mathscr{I}}

\DeclareMathOperator{\id}{id}

\DeclareMathOperator{\Cl}{Cl}

\DeclareMathOperator{\Hom}{Hom}

\DeclareMathOperator{\im}{im}

\DeclareMathOperator{\Spec}{Spec}
\DeclareMathOperator{\Max}{Max}

\DeclareMathOperator{\PP}{P}

\newcommand{\frakB}{\mathfrak B}

\newcommand{\frakc}{\mathfrak c}

\newcommand{\frakb}{\mathfrak b}
\newcommand{\frakm}{\mathfrak m}
\newcommand{\frakM}{\mathfrak M}

\newcommand{\fraka}{\mathfrak a}

\newcommand{\frakQ}{\mathfrak Q}

\newcommand{\frakp}{\mathfrak p}
\newcommand{\frakq}{\mathfrak q}

\newcommand{\frakP}{\mathfrak P}

\newcommand{\frakd}{\mathfrak d}

\newcommand{\dlim}{\varinjlim}

\theoremstyle{plain}
\newtheorem{theorem}{Theorem}[section]
\newtheorem{corollary}[theorem]{Corollary}
\newtheorem{lemma}[theorem]{Lemma}

\newtheorem{proposition}[theorem]{Proposition}

\theoremstyle{definition}
\newtheorem{definition}[theorem]{Definition}

\newtheorem{example}[theorem]{Example}

\newtheorem{remark}[theorem]{Remark}
\numberwithin{equation}{section}

\theoremstyle{plain}
\newtheorem*{theorem*}{Theorem}

\theoremstyle{plain}
\newtheorem*{theorem 1}{Theorem 1}

\theoremstyle{plain}
\newtheorem*{theorem 2}{Theorem 2}

\theoremstyle{plain}
\newtheorem*{theorem 3}{Theorem 3}

\theoremstyle{plain}
\newtheorem*{theorem 4}{Theorem 4}

\begin{document}
\title{Arithmetic Functions on a Dedekind Domain}
\author[Andrew Phillips]{Andrew Phillips}

\address{Department of Mathematics and Physical Sciences, College of Idaho, Caldwell, ID 83605}
\email{aphillips1@collegeofidaho.edu}

\maketitle

\begin{abstract}
We study functions from a unique factorization monoid to a field. The set of all such functions is a commutative ring isomorphic to a ring of formal power series over the field, with indeterminates indexed by the prime elements of the monoid. The set of all totally multiplicative functions on the monoid of integral ideals in a Dedekind domain has a ringed space structure, which, after identifying functions with the same prime ideal zeros, determines the Dedekind domain up to isomorphism.
\end{abstract}

\section{Introduction}
An arithmetic function is any function $\ZZ^+ \to \CC$, but it usually contains some kind of number-theoretic information, like the Euler-$\varphi$ function $\varphi(n) = \#(\ZZ/n\ZZ)^\times$ or the divisor functions $\sigma_\alpha(n) = \sum_{d \mid n}d^\alpha$. Let $\AAa$ be the set of all arithmetic functions. If $f, g \in \AAa$ then the sum $f + g$ is defined pointwise and there is a product $f \ast g$, called the Dirichlet product, given by
$$
(f\ast g)(n) = \sum_{d \mid n}f(d)g(n/d),
$$
where the sum is over all positive divisors of $n$. Then $\AAa$ is a commutative ring with multiplicative identity $I$ defined by $I(1) = 1$ and $I(n) = 0$ for all $n > 1$ (see \cite[Chapter 2]{Apostol}).

The structure of the ring $\AAa$ was first described in \cite{CE}. Let $\CC\llbracket X_1, X_2, \ldots\rrbracket$ be the ring of formal power series in a countable number of indeterminates $X_i$. Elements of $\CC\llbracket X_1, X_2, \ldots\rrbracket$ are formal sums
$$
\sum_{k=1}^{\infty}\alpha_kX_1^{n_1(k)}X_2^{n_2(k)}\cdots,
$$
where $\alpha_k \in \CC$ and $n_i(k) \geq 0$ are integers, all but finitely many equal to zero for each index $k$. So the homogeneous part of a given degree of such a series may have infinitely many terms, but each monomial only contains finitely many variables. Now let
$p_k$ be the $k$-th prime number, in ascending order, and define a map $\Phi : \AAa \to\CC\llbracket X_1, X_2, \ldots\rrbracket$ by
$$
\Phi(f) = \sum_{k = p_1^{n_1}p_2^{n_2}\cdots \in \ZZ^+}f(k)X_1^{n_1}X_2^{n_2}\cdots.
$$
It is shown in \cite[Section 14]{CE} that $\Phi$ is an isomorphism of rings. It is also proved in \cite{CE} that $\AAa \cong \CC\llbracket X_1, X_2, \ldots\rrbracket$ is a UFD.

In the first part of this paper we consider the following generalization of this. Let $M$ be a unique factorization monoid, written multiplicatively, satisfying $M^\times = \{1\}$, and let $K$ be a field. An \textit{arithmetic function} on $M$ is any function $M \to K$. Let $\AAa(M)$ be the set of all arithmetic functions on $M$. Just like for $\AAa$, if $f, g \in \AAa(M)$, the sum $f + g$ is defined pointwise and $f \ast g$ is defined by
$$
(f \ast g)(a) = \sum_{d \mid a}f(d)g(d^{-1}a).
$$
As before, $\AAa(M)$ is a commutative ring with identity $e$ defined by $e(1) = 1$ and $e(a) = 0$ for all $a \neq 1$. Let $\{X_p\}$ be a collection of indeterminates, one for each prime element of $M$, and let $K\llbracket (X_p)_p\rrbracket$
be the power series ring in the $X_p$'s, defined as above in the countable case. Our first main result is the following.

\begin{theorem}
There is an isomorphism of $K$-algebras $\AAa(M) \to K\llbracket (X_p)_p\rrbracket$ which is functorial in $M$. In particular, $\AAa(M)$ is a {\upshape UFD}. Also, there is a discrete valuation on $\AAa(M)$ making it into a complete metric space.
\end{theorem}

We make precise below what the functorial condition means. Now let $A$ be a Dedekind domain and $I_A$ be the set of all nonzero integral ideals of $A$. We will write $\AAa(A)$ for $\AAa(I_A)$. The second part of this paper is concerned with totally multiplicative arithmetic functions on $A$, that is, any nonzero $f \in \AAa(A)$ satisfying $f(\fraka\fraka') = f(\fraka)f(\fraka')$ for all $\fraka, \fraka' \in I_A$. Let $X = \MMM(A)$ be the set of all totally multiplicative arithmetic functions on $A$. It is natural to ask: to what extent does $X$ determine $A$? To answer this, we will develop a theory similar to that of affine varieties or schemes for the space $X$. There is a Zariski-type topology on $X$ with closed sets $$\VVV(S) = \{f \in X : \text{$f((a)) = 0$ for all $a \in S$}\},$$ where $S \subset A \sm \{0\}$. Analogous to affine schemes, the open sets
$\DDD(a) = X \sm \VVV(\{a\})$ for $a \in A \sm\{0\}$ form a basis for this topology. 

\begin{theorem}
There is a sheaf of rings $\OOO_X$ on $X$ determined by $\OOO_X(\DDD(a)) = A[1/a]$.
\end{theorem}

One limitation of this ringed space is that the stalk $\OOO_{X, f}$ is a local ring only in the case where $f$ has at most one prime ideal zero. The space $X$ is too large to 
determine $A$, so we pass to a quotient defined in the following way. Define an equivalence relation on $X$ by $f \sim g$ if $f$ and $g$ have the same prime ideal zeros, that is, $f(\frakp) = 0$ if and only if $g(\frakp) = 0$ for all nonzero $\frakp \in \Spec(A)$. Let 
$Y = \MMM_1(A) = X/\hspace{-1.2mm}\sim$, give $Y$ the quotient topology, and let
$\OOO_Y = \pi_\ast\OOO_X$, where $\pi : X \to Y$ is the quotient map. We will show
$\OOO_{Y, \pi(f)} = \OOO_{X, f}$ for any $f \in X$, so again, in general, $(Y, \OOO_Y)$ is not locally ringed. If $\varphi : A \to B$ is a homomorphism of Dedekind domains, there is a corresponding continuous map $\varphi^\ast : \MMM(B) \to \MMM(A)$ given by composing 
$f \in \MMM(B)$ with the extension of ideals from $A$ to $B$. This map $\varphi^\ast$ passes to the quotient to give a continuous map $\MMM_1(B) \to \MMM_1(A)$, so $\MMM_1$ defines a functor from the category of Dedekind domains to the category of sets.

The ringed space $(\MMM_1(A), \OOO_{\MMM_1(A)})$ determines $A$ in the following sense. Let $\mathbf{A}$ be the category whose objects are Dedekind domains and whose morphisms $A \to B$ are injective, integral ring homomorphisms. Let $\mathbf{B}$ be the category whose objects are ringed spaces $(X, \OOO_X)$ and whose morphisms are morphisms of
ringed spaces $(\Phi, \Phi^{\#}) : (X, \OOO_X) \to (Y, \OOO_Y)$ such that
$\Phi^{\#}(Y) : \OOO_Y(Y) \to \OOO_X(X)$ is injective and integral, and such that for each 
$x \in X$, the stalk map $\Phi^{\#}_x : \OOO_{Y, \Phi(x)} \to \OOO_{X, x}$ satisfies the following substitute for the usual local ring condition: for each maximal ideal $\frakm \subset
\OOO_{Y, \Phi(x)}$, the ideal $\Phi_x^{\#}(\frakm)\OOO_{X, x}$ is proper.

\begin{theorem}
The functor $\MMM_1 : \mathbf{A} \to \mathbf{B}$ is fully faithful.
\end{theorem}

In particular, if $\varphi: A \to B$ is an injective, integral homomorphism of Dedekind domains such that the corresponding morphism $\MMM_1(B) \to \MMM_1(A)$ is an isomorphism of ringed spaces, then
$\varphi$ is an isomorphism of rings.

\vspace{5mm}

\noindent\textbf{Notation.} All rings considered are commutative. If $A$ is a ring, we write $A^\ast$ for $A \sm \{0\}$ and $A^\times$ for the units of $A$. For a ring $A$ and $x \in A$, we write $A_x$ for the localization $S^{-1}A$, where $S = \{x^n : n \geq 0\}$. If $A$ is a domain, we write $\text{Frac}(A)$ for the fraction field of $A$. We write
$\Max(A)$ for the set of all maximal ideals in a ring $A$, so $\Max(A)  = \Spec(A)\sm\{(0)\}$ if $A$ is a Dedekind domain. We do not consider a field to be a Dedekind domain. Let $K$ be a field, fixed throughout the paper.

\vspace{3mm}

\noindent \textbf{Acknowledgment.} I would like to thank the anonymous referee for helpful suggestions that led to additional results and improvements to the presentation of the paper.

\section{Arithmetic functions}

We begin by reviewing some terminology related to monoids. Let $M$ be a commutative monoid, written multiplicatively, with identity element $1$, and write $M^\times$ for its set of units. For $a, d \in M$, we write $d \mid a$ if $a = bd$ for some $b \in M$. We say
$a \in M \sm M^\times$ is \textit{irreducible} if $a = bc$ for some $b, c \in M$ implies $b$ or $c$ is a unit. An element $p \in M$ is called
\textit{prime} if $p \mid ab$ implies $p \mid a$ or $p \mid b$. A \textit{unique factorization monoid} is a commutative monoid $M$ such that any $a \in M \sm M^\times$ has a factorization $a = b_1\cdots b_r$ for some irreducible elements $b_i \in M$ and if $a = c_1\cdots c_s$ is another such factorization, then $r = s$ and, after reordering, $b_i = u_ic_i$ for some $u_i \in M^\times$. In a unique factorization monoid with $M^\times = \{1\}$, an element is prime if and only if it is irreducible. We will usually use the term ``prime" when discussing a factorization of an element into irreducibles. If $M$ is a unique factorization monoid, we say $a, b \in M$ are \textit{relatively prime} if they have no common prime factors.

\begin{example}
If $R$ is a UFD then $(R^\ast, \cdot)$ is a unique factorization monoid. If $A$ is a Dedekind domain, the set $I_A$ of all nonzero integral ideals in $A$ is a unique factorization monoid. This is the most important example for what follows.
\end{example}

\begin{remark}
To simplify terminology and avoid complications with units, we make the following assumption. In what follows, by ``monoid" we mean a unique factorization monoid $M$, written multiplicatively, with $M^\times = \{1\}$. We also assume every monoid homomorphism $\varphi : M \to N$ satisfies $\varphi(1) = 1$. We write $\PP(M)$ for the set of all prime elements in $M$.
\end{remark}

\begin{definition}
Let $M$ be a monoid. An \textit{arithmetic function} on $M$ is any function $M \to K$.
An arithmetic function $f$ is \textit{multiplicative} if $f(ab) = f(a)f(b)$ for all $a, b \in M$ relatively prime and $f(1) = 1$, and $f$ is \textit{totally multiplicative} if $f(ab) = f(a)f(b)$ for all $a, b \in M$ and $f(1) = 1$.
\end{definition}

Note that a totally multiplicative function $f$ is determined by the values $f(p)$ for $p \in \PP(M)$, so we will often define such an $f$ by just specifying these values,
with the understanding that $f(1) = 1$. Also, the condition $f(1) = 1$ for a multiplicative function $f$ is equivalent to assuming $f$ is not identically $0$. There are certain situations when we will need to evaluate an arithmetic function $f : I_A \to K$, with $A$ a Dedekind domain, at $(0)$, for example when extending an ideal from one Dedekind domain to another, so we define $f((0)) = 0$, when applicable. 
The classical definition of an arithmetic function is the case where $M = \ZZ^+$.

\begin{example} Let $F$ be a number field with ring of integers $\Oo_F$. The function $\varphi_F(\fraka) = \#(\Oo_F/\fraka)^\times$ is a multiplicative arithmetic function on $I_{\Oo_F}$. The ideal norm $\text{N}(\fraka) = \#(\Oo_F/\fraka)$ is totally multiplicative. 
\end{example}

\begin{example}\label{mobius}
Let $M$ be a monoid. Define the M\"obius function $\mu$ by $\mu(1) = 1$ and 
for $a = p_1\cdots p_r \in M$, $\mu(a) = (-1)^r$ if all the primes $p_i$ are distinct, and
$\mu(a) = 0$ otherwise. Then $\mu$ is multiplicative but not totally multiplicative. 
\end{example}

\begin{example}
Let $A$ be a Dedekind domain. The function 
$$
d(\fraka) = \sum_{\frakd \mid \fraka} 1
$$
is multiplicative, where the sum runs over all $\frakd \in I_A$ dividing $\fraka$. This is because if $\fraka$ and $\fraka'$ are relatively prime, then there is a bijection between divisors $\frakd \mid \fraka\fraka'$ and pairs of divisors $\frakb \mid \fraka$ and $\frakc \mid \fraka'$, given by
$\frakd \mapsto (\frakd + \fraka, \frakd + \fraka')$, and in the other direction by $(\frakb, \frakc) \mapsto \frakb\frakc$.
\end{example}

\begin{example}
Let $A = \Oo_F$ for some number field $F$. For each $\alpha \in K$, the function
$$
\sigma_\alpha(\fraka) = \sum_{\frakd \mid \fraka} \text{N}(\frakd)^\alpha
$$
is multiplicative for the same reason as the previous example, using that the ideal norm is multiplicative.
\end{example} 

We write $\AAa(M)$ for the set of all arithmetic functions on a monoid $M$ and $\MMM(M)$ for the set of all totally multiplicative arithmetic functions on $M$. Suppose $\varphi : M \to N$ is a monoid homomorphism. This leads to a map $\varphi^\ast : \AAa(N) \to \AAa(M)$ given by $\varphi^\ast(g) = g\circ \varphi$. In the case of ideals in a Dedekind domain, there is one technical point. If $\theta : A \to B$ is a homomorphism of Dedekind domains, there is an induced multiplicative map $\theta_{\text{e}} : 
I_A  \to I_B\cup\{(0)\}$ given by extending ideals $\theta_{\text{e}}(\fraka) = \theta(\fraka)B$. We still get a map $\theta_{\text{e}}^\ast : \AAa(I_B) \to \AAa(I_A)$ as before, using the convention that $g((0)) = 0$ for any $g \in \AAa(I_B)$.

Suppose $\varphi: M \to N$ and $\psi : N \to L$ are homomorphisms between monoids. Then 
$(\psi \circ \varphi)^\ast = \varphi^\ast \circ \psi^\ast$ and $\text{id}_M^\ast = \text{id}_{\AAa(M)}$, so we obtain a contravariant functor $\AAa$ from the category of monoids to the category of sets. If $a, b \in M$ 
and $g \in \MMM(N)$, then $\varphi^\ast(g)(ab) = (\varphi^\ast(g)(a))(\varphi^\ast(g)(b))$. Also, $\varphi^\ast(g)(1) = g(1) = 1$, so $\varphi^\ast$ restricts to a map
$\MMM(N) \to \MMM(M)$. Since extension of ideals is multiplicative, any homomorphism of Dedekind domains $A \to B$ induces a map
$\MMM(I_B) \to \MMM(I_A)$.

\section{Ring structure}

Let $M$ be a monoid. Define addition in $\AAa(M)$ pointwise: $(f+g)(a) = f(a) + g(a)$ for $f, g \in \AAa(M)$ and $a \in M$. Then $(\AAa(M), +)$ is an abelian group with identity the zero function and inverse of $f \in
\AAa(M)$ equal to $-f$. If $d \mid a$ in $M$, so $a = bd$ for some $b \in M$, we write $d^{-1}a$ for $b$.

\begin{definition}
Define the \textit{Dirichlet product} $f\ast g$ of $f, g \in \AAa(M)$ by
$$
(f \ast g)(a) = \sum_{d \mid a} f(d)g(d^{-1}a) = 
\sum_{bc = a} f(b)g(c),
$$
where the second sum is over all pairs $(b, c) \in M \times M$ such that $bc = a$.
\end{definition}

Also, define $e \in \AAa(M)$ by
$$
e(a) = 
\left\{\begin{array}{ll}
1 & \text{if $a = 1$} \\
0 & \text{if $a \neq 1$}.
\end{array} \right.
$$
The basic properties of the Dirichlet product are summarized in the following result.

\begin{lemma} Let $f, g, h \in \AAa(M)$ and $e \in \AAa(M)$ be defined as above.\\
{\upshape (a)} $f \ast g = g \ast f$ \\
{\upshape (b)} $f\ast(g\ast h) = (f\ast g)\ast h$ \\
{\upshape (c)} $f\ast e = e\ast f = f$ \\
{\upshape (d)} $f\ast(g + h) = (f\ast g) + (f\ast h)$ \\
{\upshape (e)} $(f + g)\ast h = (f\ast h) + (g\ast h)$
\end{lemma}

\begin{proof} 
Each is an easy calculation, identical to the classical case of $M = \ZZ^+$. See \cite[Section 2.6]{Apostol}.
\end{proof}

\begin{proposition}
The triple $(\AAa(M), +, \ast)$ is a commutative, local ring with multiplicative identity $e$, and is a $K$-algebra via pointwise scaling. Its units are $\AAa(M)^\times = \{f \in \AAa(M) : f(1) \neq 0\}$.
\end{proposition}

\begin{proof}
That $\AAa(M)$ is a commutative $K$-algebra follows from the discussion above. To prove the statement about units, first suppose $f \in \AAa(M)^\times$, so there exists $g \in \AAa(M)$ such 
that $f \ast g = e$. Then
$$
1 = e(1) = \sum_{bc = 1}f(b)g(c) = f(1)g(1),
$$
so $f(1) \neq 0$. Conversely, suppose $f(1) \neq 0$. Using induction on the number of prime divisors of $a$, the formulas $g(1) = 1/f(1)$ and 
$$
g(a) = -\frac{1}{f(1)}\sum_{\substack{d \mid a \\ d \neq a}} g(d)f(d^{-1}a)
$$
uniquely define a function $g$. The proof that $f\ast g = e$ is the same as the case of $A = \ZZ^+$. See \cite[Theorem 2.8]{Apostol}. Finally, to show $\AAa(M)$ is a local ring, let $\frakm = \{f \in \AAa(M) : f(1) = 0\}$, which is easily seen to be an ideal. Since $\AAa(M)^\times = \AAa(M) \sm \frakm$, the ring $\AAa(M)$ is
local with unique maximal ideal $\frakm$.
\end{proof}

We will denote the multiplicative inverse of $f \in \AAa(M)^\times$ by $f^{-1}$.

\begin{proposition}\label{mult}\
\begin{enumerate}
\item[(a)] If $f, g \in \AAa(M)$ are both multiplicative, then $f\ast g$ is multiplicative. 
\item[(b)] If $f \in \AAa(M)^\times$ is multiplicative, then $f^{-1}$ is multiplicative.
\end{enumerate}
\end{proposition}

\begin{proof}
The proof of (a) is the same as the classical case (\cite[Theorem 2.14]{Apostol}). The proof of
(b) is an easy adaptation of the classical proof in \cite[Theorem 2.16]{Apostol}, where the condition ``choose a pair $m$ and $n$ with $(m, n) = 1$ and $f(mn) \neq f(m)f(n)$ for which $mn$ is as small as possible" there is replaced with ``choose $a$ and $b$ relatively prime and $f(ab) \neq f(a)f(b)$ for which $ab$ has the fewest number of prime divisors".
\end{proof}

\begin{corollary}
If $K$ has characteristic $0$ then
the set $\MMM_0(M)$ of all multiplicative arithmetic functions on $M$ is a torsion-free subgroup of
$\AAa(A)^\times$.
\end{corollary}

\begin{proof}
Any multiplicative $f$ satisfies $f(1) = 1$, so $f \in \AAa(M)^\times$. Also, the identity $e$ is multiplicative, so $\MMM_0(M)$ is a subgroup of $\AAa(M)^\times$. To show it is torsion-free, suppose $f \in \MMM_0(M)$ and $f^n = e$ for some $n \geq 1$. Suppose there exists $p \in \PP(M)$ such that $f(p^k) \neq 0$ for some $k \geq 1$ and let $\ell$ be the smallest such power. Then
$$
0 = f^{n}(p^\ell) = \sum_{b_1\cdots b_n = p^\ell}f(b_1)\cdots
f(b_n) = n f(p^\ell)f(1)^{n -1} = n f(p^\ell),
$$
a contradiction. Hence $f(p^k) = 0$ for all $p$ and all $k \geq 1$, which means $f = e$ since it is multiplicative.
\end{proof} 

\begin{example} From Proposition \ref{mult}(a), if $f \in \AAa(M)$ is multiplicative, then the function
$$
g(a) = \sum_{d \mid a}f(d)
$$
is also multiplicative, since $g = f\ast u$, where $u(a) = 1$ for all $a$ is multiplicative.
\end{example}

\begin{example}
Recall the M\"obius function $\mu$ from Example \ref{mobius}. An easy calculation shows
$\sum_{d \mid a}\mu(d) = 0$ for all $a \neq 1$, so $\mu \ast u = e$, with $u$ as above. If $g = f \ast u$ for some $f \in \AAa(M)$, then $g \ast \mu = f \ast e = f$, which proves the M\"obius inversion formula: for any $f \in \AAa(M)$, define $g \in \AAa(M)$ by $g(a) = \sum_{d \mid a}f(d)$.
Then
$$
f(a) = \sum_{d \mid a}\mu(d^{-1}a)g(d).
$$
\end{example}

\begin{proposition}\label{iso}
If $\varphi : M \to N$ is an isomorphism of monoids, the induced map
$\varphi^\ast : \AAa(N) \to \AAa(M)$ is an isomorphism of local $K$-algebras.
\end{proposition}

\begin{proof}
Clearly $\varphi^\ast$ is $K$-linear. Next, since $\varphi$ is an isomorphism, given any $a \in M$, there is a multiplicative one-to-one correspondence between pairs $(b, c) \in M \times M$ such that $bc = a$ and pairs
$(x, y) \in N \times N$ such that $xy = \varphi(a)$, given by
$(b, c) \mapsto (\varphi(b), \varphi(c))$, and in the other direction,
$(x, y) \mapsto (\varphi^{-1}(x), \varphi^{-1}(y))$. It follows that for any 
$f, g \in \AAa(N)$ and $a \in M$,
\begin{align*}
[\varphi^\ast(f\ast g)](a) &= (f \ast g)(\varphi(a)) 
= \sum_{xy = \varphi(a)}
f(x)g(y) \\
&= \sum_{\varphi(b)\varphi(c) = \varphi(a)}f(\varphi(b))g(\varphi(c)) \\
&= \sum_{bc = a}(\varphi^\ast f)(b)(\varphi^\ast g)(c) 
= [(\varphi^\ast f)\ast
(\varphi^\ast g)](a),
\end{align*}
so $\varphi^\ast(f \ast g) = (\varphi^\ast f)\ast(\varphi^\ast g)$. We also have $\varphi^\ast(e_ N) = e_M$, so $\varphi^\ast$ is a 
ring homomorphism. From the functorial properties $(\varphi \circ \varphi^{-1})^\ast = (\varphi^{-1})^\ast \circ \varphi^\ast$ and $\id_N^\ast = \id_{\AAa(N)}$, and the similar properties in the other order,
we find $\varphi^\ast$ is an isomorphism with inverse $(\varphi^{-1})^\ast$. Clearly $\varphi^\ast(\frakm_{\AAa(N)}) = \frakm_{\AAa(M)}$, which means $\varphi^\ast$ is a local isomorphism. 
\end{proof}

Given a collection of indeterminates $(X_p)_{p \in \PP(M)}$, one for each 
$p \in \PP(M)$, let $K\llbracket (X_p)_p\rrbracket$ be the ring of formal power series in the variables $(X_p)_p$, with coefficients in $K$. Here we mean the power series ring in the sense of \cite[III, \S 2.10]{Bourbaki1} and \cite[IV, \S 4]{Bourbaki2}, whose elements are defined as follows. Let $\ZZ_{\geq 0}$ be the set of nonnegative integers and $\ZZ_{\geq 0}^{(\PP(M))}$ be the set of all functions $n : \PP(M) \to \ZZ_{\geq 0}$, with $n_p = 0$ for all but finitely many $p 
\in \PP(M)$, where we write $n_p$ for the value at $p$. We denote these functions by $(n_p)$. For $(n_p) \in \ZZ_{\geq 0}^{(\PP(M))}$, if
$n_p > 0$ for $p \in \{p_1, \ldots, p_r\}$, we use the notation
$$
(X_p)^{(n_p)} = X_{p_1}^{n_{p_1}}\cdots X_{p_r}^{n_{p_r}}.
$$
Elements of $K\llbracket (X_p)_p\rrbracket$ are power series
$$
\sum_{(n_p) \in \ZZ_{\geq 0}^{(\PP(M))}} \alpha_{(n_p)} (X_p)^{(n_p)}
$$
where $\alpha_{(n_p)} \in K$. Note that each monomial term contains only finitely many variables, but there may be infinitely many terms of a given degree.

\begin{theorem}\label{isomorphism}
There is an isomorphism of $K$-algebras $\Phi : \AAa(M) \to K\llbracket (X_p)_p\rrbracket$, which is functorial in $M$ in the following sense: if $\varphi : M \to N$ is an isomorphism of monoids, then there is an isomorphism
of $K$-algebras $\psi : K\llbracket (Y_q)_{q \in \PP(N)}\rrbracket
\to K\llbracket (X_p)_{p \in \PP(M)}\rrbracket$ such that the following diagram commutes.
\begin{equation}\label{diagram0}
\xymatrix{
\AAa(N) \ar[rr]^>>>>>>>>>>>>{\varphi^{\ast}} \ar[d]_{\Phi_N}  & & \AAa(M) \ar[d]^{\Phi_M} \\
K\llbracket (Y_q)_q \rrbracket \ar[rr]^>>>>>>>>>>{\psi} && K\llbracket (X_p)_p  \rrbracket
}
\end{equation}
\end{theorem}

\begin{proof}
For $a \in M$ not equal to $1$, define $n(a) \in \ZZ_{\geq 0}^{(\PP(M))}$
by $n(a)_p = \text{ord}_p(a)$, the exponent of $p$ in the prime factorization of $a$. Also, let $n(1) \in \ZZ_{\geq 0}^{(\PP(M))}$ be the zero function.
Given $f \in \AAa(M)$, define a power series $\Phi(f) \in K\llbracket(X_p)_p\rrbracket$
by
$$
\Phi(f) = \sum_{a \in M} f(a)(X_p)^{n(a)}.
$$
Clearly $\Phi : \AAa(M) \to K\llbracket(X_p)_p\rrbracket$ is injective, $K$-linear, and $\Phi(e) = 1$, where $e \in \AAa(M)$ is the multiplicative identity. Next, based on how multiplication is defined in the ring $K\llbracket(X_p)_p\rrbracket$ (see \cite[III, \S\S 10-11]{Bourbaki1}), 
$$
\Phi(f\ast g) = \sum_{a \in M}(f \ast g)(a)(X_p)^{n(a)}
= \sum_{a \in M}\sum_{bc = a} f(b)g(c)(X_p)^{n(a)}
= \Phi(f)\Phi(g),
$$
so $\Phi$ is a ring homomorphism. Finally, let
$$
F = \sum_{(n_p)} \alpha_{(n_p)}(X_p)^{(n_p)} \in K\llbracket(X_p)_p\rrbracket.
$$
Given $(n_p) \in \ZZ_{\geq 0}^{(\PP(M))}$, there is a corresponding vector
$(n_{p_1}, \ldots, n_{p_r}) \in (\ZZ^+)^r$, just the nonzero values of $n$ in some order. 
Define an element $a_{(n_p)} = p_1^{n_{p_1}}\cdots p_r^{n_{p_r}} \in M$. 
Every $a \in M$ arises as $a_{(n_p)}$ for a unique $(n_p) = n(a) \in 
\ZZ_{\geq 0}^{(\PP(A))}$. Define $f \in \AAa(M)$ by $f(a) = \alpha_{n(a)}$. Then
$\Phi(f) = F$ and hence $\Phi$ is surjective.

Now suppose $\varphi : M \to N$ is an isomorphism of monoids, so $\varphi^\ast : \AAa(M) \to \AAa(N)$ is an isomorphism of $K$-algebras. Given a function $m : \PP(N) \to \ZZ_{\geq 0}$ with finite support, there is a corresponding function $m^\ast : \PP(M) \to \ZZ_{\geq 0}$ with finite support given by $m^\ast_p = m_{\varphi(p)}$. Define
$\psi :  K\llbracket(Y_q)_q\rrbracket \to K\llbracket(X_p)_p\rrbracket$ by
$$
\psi\Big(\sum_{(m_q)}\beta_{(m_q)}(Y_q)^{(m_q)}\Big) = 
\sum_{(m_q)}\beta_{(m_q)}(X_{\varphi^{-1}(q)})^{(m^\ast_{\varphi^{-1}(q)})}.
$$
Note that since $\varphi$ is an isomorphism, for $b \in N$ and $p \in \PP(M)$ we have $n(b)^\ast_p = \text{ord}_{\varphi(p)}(b) = \text{ord}_p(\varphi^{-1}(b))$, so $n(b)^\ast = n(\varphi^{-1}(b))$. Then for any
$f \in \AAa(N)$,
$$
\psi(\Phi_N(f)) = \sum_{b \in N}f(b)(X_{\varphi^{-1}(q)})^{n(\varphi^{-1}(b))}
= \sum_{a \in M}f(\varphi(a))(X_p)^{n(a)} = \Phi_M(\varphi^\ast f).
$$
This shows (\ref{diagram0}) commutes, and hence $\psi$ is an isomorphism of $K$-algebras.
\end{proof}

\begin{corollary}\label{UFD}
The ring $\AAa(M)$ is a {\upshape UFD} of dimension $\#\PP(M)$, where the dimension is $\infty$ if $\PP(M)$ is infinite. The ring is Noetherian if and only if $\PP(M)$ is finite.
\end{corollary}

\begin{proof}
If $\#\PP(M) = n$ then $\AAa(M) \cong K\llbracket X_1, \ldots, X_n\rrbracket$, which is 
a Noetherian UFD of dimension $n$ (see \cite[IV, Theorem 9.3, Corollary 9.5]{Lang} and \cite[Theorem 15.4]{Matsumura}). If $\PP(M)$ is infinite, then 
$K\llbracket(X_p)_p\rrbracket$ is a UFD by \cite[Theorem 1]{Nishimura}. This ring is infinite dimensional since, choosing a sequence of nonzero prime elements $\{p_i : i \geq 1\}$ in $M$, there is an infinite chain of prime ideals
$$
(X_{p_1}) \subsetneq (X_{p_1}, X_{p_2}) \subsetneq (X_{p_1}, X_{p_2}, 
X_{p_3}) \subsetneq \cdots
$$
in $K\llbracket(X_p)_p\rrbracket$. As $\AAa(M) \cong K\llbracket(X_p)_p\rrbracket$ is local and infinite dimensional, it is not Noetherian.
\end{proof}

\section{Topological ring structure}

Let $M$ be a monoid. In this section we will define a valuation on the ring $\AAa(M)$ making it into a complete metric space. We start by defining a valuation on $K\llbracket (X_p)_p\rrbracket$. If $(n_p)  \in \ZZ_{\geq 0}^{(\PP(M))}$ is nonzero at the primes $\{p_1, \ldots, p_r\}$, define $\ell(n_p) = n_{p_1}+\cdots + n_{p_r}$. Define a function 

$$
w : K\llbracket (X_p)_p\rrbracket^\ast \to \ZZ_{\geq 0}, \quad w\Big(\sum_{(n_p)}\alpha_{(n_p)}(X_p)^{(n_p)}\Big) = \min\{\ell(n_{p})
: \alpha_{(n_p)} \neq 0\}
$$
and set $w(0) = \infty$.
In other words, $w$ gives the smallest degree of a homogeneous part appearing in a power series. It is easily verified that $w(FG) = w(F) + w(G)$ and $w(F+G) \geq \min(w(F), w(G))$, so
$w$ determines a discrete valuation on $\text{Frac}(K\llbracket (X_p)_p\rrbracket)$.
Let $\Phi : \AAa(M) \to K\llbracket (X_p)_p\rrbracket$ be the isomorphism from Theorem \ref{isomorphism}, and for $a = p_1\cdots p_r \in M$, where the primes $p_i$ are not necessarily distinct, define $\lambda(a) = r$ and $\lambda(1) = 0$. Recall that for $a \in M$, the 
function $n(a) : \PP(M) \to \ZZ_{\geq 0}$ is given by $n(a)_p = \text{ord}_p(a)$,
so $\ell(n(a)) = \lambda(a)$. Hence, for any $f \in \AAa(M)$,
$$
w(\Phi(f)) = \min\{\lambda(a) : a \in M, f(a) \neq 0\}
$$
determines a discrete valuation $v$ on $\text{Frac}(\AAa(M))$. Note that $\AAa(M)$ is contained in the valuation ring of $v$, but is not equal to it if $\#\PP(M) > 1$ by Corollary \ref{UFD}.

Fix a real number $c \in (0, 1)$ and for $F \in \text{Frac}(\AAa(M))^\ast$, set $|F| = c^{v(f)}$ and $|0| = 0$, which is a non-archimedean absolute value on $\text{Frac}(\AAa(M))$. A function $f \in \AAa(M)$ is in $\AAa(M)^\times$ if and only if $f(1) \neq 0$, which is equivalent to
$|f| = 1$, so $\AAa(M)^\times = \{f \in \AAa(M) : |f| = 1\}$. The absolute value $|\cdot|$ determines a metric $d$ on $\AAa(M)$ in the usual way: $d(f, g) = |f - g|$. The metric topology defined
on $\AAa(M)$ in this way is independent of the choice of the number $c$ and makes $\AAa(M)$ into a topological ring. Defining an analogous metric on $K\llbracket (X_p)_p\rrbracket$ using the valuation $w$, it follows from the definition of $v$ that the isomorphism $\Phi : \AAa(M) \to K\llbracket (X_p)_p\rrbracket$ is an isometry.

\begin{proposition}
If $\varphi : M \to N$ is an isomorphism of monoids, then $\varphi^\ast : \AAa(N) \to \AAa(M)$ is an isometric isomorphism of topological rings.
\end{proposition} 

\begin{proof}
By Proposition \ref{iso}, $\varphi^\ast : \AAa(N) \to \AAa(M)$ is an isomorphism of rings. Let $v_M$ and $v_N$ be the discrete valuations on $\AAa(M)$ and $\AAa(N)$, respectively, as above. Since $\varphi$ is an isomorphism, each $b \in N$ is of the form
$b = \varphi(a)$ for a unique $a \in M$, and for each $p \in \PP(M)$,
we have $\varphi(p) \in \PP(N)$. This implies $\lambda(\varphi(a)) = \lambda(a)$ and for any $f \in \AAa(N)$,
\begin{align*}
v_M(\varphi^\ast f) &= \min\{\lambda(a) : a \in M, f(\varphi(a)) \neq 0\} \\
&= \min\{\lambda(\varphi(a)) : a \in M, f(\varphi(a)) \neq 0\} \\
&= \min\{\lambda(b) : b \in N, f(b) \neq 0\} \\
&= v_N(f).
\end{align*}
If $d_M$ and $d_N$ are the metrics on $\AAa(M)$ and $\AAa(N)$ as above, this shows
$d_N(f, g) = d_M(\varphi^\ast f, \varphi^\ast g)$ for any $f, g \in \AAa(N)$. Therefore $\varphi^\ast$ is an isometry, and thus continuous. Similarly, $(\varphi^\ast)^{-1} = (\varphi^{-1})^\ast$ is a continuous isometry.
\end{proof}

Let 
$$
\frakm = \{f \in \AAa(M) : f(1) = 0\} = \{f \in \AAa(M) : v(f) \geq 1\}
$$
be the maximal ideal in $\AAa(M)$.

\begin{lemma}
We have $\displaystyle{\bigcap_{n = 1}^{\infty} \frakm^n = (0)}$.
\end{lemma}

\begin{proof}
Suppose $g \in \bigcap_{n = 1}^\infty \frakm^n$ is nonzero,
so there exists $a_g \in M$ such that $g(a_g) \neq 0$. Let $m = \lambda(a_g)$ and choose an integer $n > m$. If $a_g = b_1\cdots b_n$ for some $b_i \in M$, then $\sum_{i=1}^n\lambda(b_i) = m$. Since $m < n$, some $b_i = 1$, which implies that for any $f_1, \ldots, f_n \in \frakm$,
$$
(f_1\ast\cdots\ast f_n)(a_g) = \sum_{b_1\cdots b_n = a_g}f_1(b_1)\cdots
f_n(b_n) = 0.
$$
In particular, $g(a_g) = 0$, a contradiction.

An alternative approach is to use the valuation $v$. If $f \in \frakm^n$ then
$$
f = \sum_{i=1}^r f_{i1}\ast\cdots \ast f_{in}
$$
for some $f_{ij} \in \frakm$. Since each $v(f_{ij}) \geq 1$, it follows that $v(f) \geq n$. Therefore if $f \in \bigcap_{n=1}^\infty\frakm^n$ then $v(f) \geq n$ for all $n \geq 1$, which implies $f = 0$.
\end{proof}

The valuation $v$ allows us to define a function $\text{N} : \AAa(M) \to \ZZ_{\geq 0}$ by $\text{N}(0) = 0$ and $\text{N}(f) = 2^{v(f)}$ for any nonzero $f \in \AAa(M)$. 
Then $\text{N}(f \ast g) = \text{N}(f)\text{N}(g)$ for any $f, g \in \AAa(M)$, and $\text{N}(f) = 1$ if and only if 
$f \in \AAa(M)^\times$. Also, since $\AAa(M)$ is a UFD, if $\text{N}(f) = 2$, that is $v(f) = 1$, then $f$ is a prime element.

For any $F \in K\llbracket (X_p)_p\rrbracket$ and integer $n \geq 0$, let $F^{(n)}$ be the homogeneous degree $n$ part of $F$. Let $$\frakM = \{F \in  K\llbracket (X_p)_p\rrbracket : F^{(0)} = 0\}$$ be the maximal ideal of $K\llbracket (X_p)_p\rrbracket$. Note that $\frakM$ contains $(X_p : p \in \PP(M))$, the ideal generated by all $X_p$. These ideals are equal if and only if $\PP(M)$ is finite.
If $\PP(M)$ is infinite, then, for example, $\sum_{p \in \PP(M)}X_p$ is in 
$\frakM$ but not $(X_p : p \in \PP(M))$. 

For each $n \geq 1$ define an ideal
$$
\frakM_n = \{F \in  K\llbracket (X_p)_p\rrbracket : w(F) \geq n\}.
$$
Then $\frakM_1 = \frakM$ and these ideals form a descending chain $K\llbracket (X_p)_p\rrbracket \supset \frakM_1 \supset \frakM_2 \supset \cdots$.
We have $\frakM^n \subset \frakM_n$ and the ideals are equal if and only if $\PP(M)$ is finite. If
$\PP(M)$ is infinite, then, for example, $\sum_{p\in\PP(M)}X_p^n$ is in
$\frakM_n$ but not $\frakM^n$.

The corresponding situation in $\AAa(M)$ is the following.
For each $n \geq 1$ define an ideal $$\frakm_n = \{f \in \AAa(M) : v(f) \geq n\},$$ so $\frakm_1 = \frakm$ and $\AAa(M) \supset \frakm_1 \supset \frakm_2 \supset \cdots$. Then $\frakm^n \subset \frakm_n$ and they are equal if and only if $\PP(M)$ is finite. There is a  topology on $\AAa(M)$ such that 
$\{f + \frakm_n : n \geq 1\}$ is a neighborhood basis of $f \in \AAa(M)$. We will call this the
$\frakm_n$-topology.

\begin{proposition} 
The $\frakm_n$-topology on $\AAa(M)$ is the same as the metric topology induced by the absolute value $|\cdot|$.
\end{proposition}

\begin{proof}
Suppose $U \subset \AAa(M)$ is open in the metric topology and let $f_0 \in U$. Then there is an $r > 0$ such that $D(f_0, r) = \{f \in \AAa(M) : |f - f_0| < r\} \subset U$. Choose an integer $n > \log_c(r)$. If
$f \in f_0 + \frakm_n$ then $v(f-f_0) > \log_c(r)$, so $|f-f_0| < r$. Hence $f_0+\frakm_n \subset D(f_0, r) \subset U$,
which shows $U$ is open in the $\frakm_n$-topology.

Conversely, suppose $U \subset \AAa(M)$ is open in the $\frakm_n$-topology and let $f_0 \in U$. Then $f_0 + \frakm_n \subset U$ for some $n \geq 1$. Choose a real number $r$ such that $\log_c(r) > n$. Then for $f \in D(f_0, r)$, $|f-f_0| < r < c^n$, which means $f \in f_0 + \frakm_n$, so $D(f_0, r) \subset U$.
This shows $U$ is open in the metric topology.
\end{proof}

\begin{proposition}
The metric space $\AAa(M)$ is complete.
\end{proposition}

\begin{proof}
Let $\{f_n\}_{n \geq 1}$ be a Cauchy sequence in $\AAa(M)$. Then for all $\varepsilon  > 0$, there is an integer $t \geq 1$ such that $|f_m - f_n| < \varepsilon$ for all $m, n \geq t$, which is equivalent to $f_m(a) = f_n(a)$ for all $a \in M$ such that $\lambda(a) < \log_c(\varepsilon)$ and all $m, n \geq t$. Hence, for all integers $k \geq 0$, there is an
integer $t(k) \geq 1$ such that $f_m(a) = f_n(a)$ for all $a \in M$ such that
$\lambda(a) = k$ and all $m, n \geq t(k)$. 

Define $f \in \AAa(M)$ as follows: for $a \in M$ with $\lambda(a) = k$, set $f(a) = 
f_{t(k)}(a)$. For any $k \geq 0$, if $n \geq \max\{t(0), \ldots, t(k)\}$ then
$$
v(f - f_n) = \min\{\lambda(a) : a \in M, f(a) \neq f_n(a)\} > k,
$$
which gives $|f - f_n| < c^k$. Therefore $\{f_n\}$ converges to $f$. 
\end{proof}

It follows that the metric space $K\llbracket (X_p)_p\rrbracket$ is also complete, since $\Phi : \AAa(M) \to K\llbracket (X_p)_p\rrbracket$ is an isometric isomorphism. To obtain a corollary of this, suppose $\PP(M) = \{p_i : i \geq 1\}$ is
countably infinite, which is the case, for example, if $M = I_{\Oo_L}$ for any number field $L$.
Then $M = \{a_i : i \geq 1\}$ is countably infinite. Choose an indexing such that 
$\lambda(a_k) \leq \lambda(a_{k+1})$ for all $k \geq 1$, which implies $\lambda(a_k) \to \infty$ as $k \to \infty$ (in the usual metric in $\RR$). Set $X_i = X_{p_i}$, and as above, if $a = p_{i_1}^{n_1}\cdots p_{i_r}^{n_r}$, we write
$$
(X_i)^{n(a)} = X_{i_1}^{n_1}\cdots X_{i_r}^{n_r}.
$$
Any $F \in K\llbracket (X_i)_i\rrbracket$ can be expressed in the form
$$
F = \sum_{k=1}^{\infty}\alpha_k(X_i)^{n(a_k)}
$$
for some $\alpha_k \in K$. Given such an $F$ and an integer $t \geq 1$, let
$$
F_t = \sum_{k=1}^t\alpha_k(X_i)^{n(a_k)},
$$
which is a polynomial in $K\llbracket (X_i)_i\rrbracket$. Then $F_t - F_{t-1} = \alpha_t(X_i)^{n(a_t)}$, so $w(F_t - F_{t-1}) = \lambda(a_t) \to \infty$ as $t \to \infty$. This shows $\{F_t\}$ is a Cauchy sequence in $K\llbracket (X_i)_i\rrbracket$ and hence converges, so we may write
$$
F = \lim_{t \to \infty}\sum_{k=1}^t\alpha_k(X_i)^{n(a_k)}.
$$
The isomorphism $\Phi : \AAa(M) \to K\llbracket (X_i)_i\rrbracket$ is given by 
$$
\Phi(f) = \sum_{k=1}^{\infty}f(a_k)(X_i)^{n(a_k)}.
$$
For each $i \geq 1$, let $\pi_i = \Phi^{-1}(X_i)$, which is a prime element of $\AAa(M)$ since
$v(\pi_i) = 1$. Then $\pi_i(p_i) = 1$ for all $i \geq 1$ and $\pi_i(a) = 0$ if $a \neq p_i$. Let
$f \in \AAa(M)$ and for each integer $t \geq 1$, let
$$
f_t = \sum_{k=1}^tf(a_k)(\pi_i)^{n(a_k)} \in \AAa(M),
$$
where, as above, $(\pi_i)^{n(a)} = \pi_{i_1}^{n_1}\ast\cdots\ast\pi_{i_r}^{n_r}$ for 
$a = p_{i_1}^{n_1}\cdots p_{i_r}^{n_r}$. Then $\Phi(f_t) = \Phi(f)_t$ and
$\Phi(f)_t \to \Phi(f)$ as $t \to \infty$ by what we showed above. Hence $\Phi(f_t) \to \Phi(f)$, which implies $f_t \to f$ as $t \to \infty$ since $\Phi$ is an isometry. This proves the following result, which generalizes \cite[Theorem 1]{SS2}.

\begin{corollary}
Suppose $\PP(M)$ is countably infinite. There is a collection of prime elements
$\{\pi_i : i \geq 1\}$ in $\AAa(M)$ such that the $K$-algebra $K[\{\pi_i\}_i]$ is dense in
$\AAa(M)$ and each $f \in \AAa(M)$ can be expressed as a convergent series
$$
f = \sum_{k=1}^{\infty}\alpha_k\pi_1^{n_1(k)}\ast\pi_2^{n_2(k)}\ast\cdots,
$$
for a unique sequence $\{\alpha_k\}$ in $K$ and some integers $n_i(k) \geq 0$, all but finitely many equal to $0$ for each $k$.
\end{corollary}

\section{Ringed space structure}

The aim of the next two sections is to define a ringed space structure on the set of all totally multiplicative functions on the monoid of integral ideals in a Dedekind domain in such a way that there is an equivalence of categories between the category of Dedekind domains and the category of such ringed spaces. This will be similar to
the development of affine schemes, but with a few key differences. 

Let $A$ be a Dedekind domain. From now on, we write $\MMM(A)$ for $\MMM(I_A)$, and if 
$\varphi : A \to B$ is a homomorphism of Dedekind domains, we write $\varphi^\ast : \MMM(B) \to \MMM(A)$ for what we would have called $\varphi_{\text{e}}^\ast$ above.
Let $X = \MMM(A)$ and for $S \subset A^\ast$, define
$$
\VVV(S) = \{f \in X : \text{$f((a)) = 0$ for all $a \in S$}\}.
$$
Clearly $\VVV$ is inclusion reversing and satisfies the following easily verified properties.

\begin{proposition}\label{properties} \
\begin{enumerate}
\item[(a)] For $S, T \subset A^\ast$, define $ST = \{xy : x \in S, y \in T\}$. Then $\VVV(S) \cup \VVV(T) = \VVV(ST)$.
\item[(b)] If $\{S_i\}_i$ is a collection of subsets of $A^\ast$, then $\bigcap_i\VVV(S_i) = \VVV(\bigcup_i S_i)$.
\end{enumerate}
\end{proposition}

We also have $\VVV(\{1\}) = \varnothing$ and $\VVV(\varnothing) = X$, so the proposition shows that the sets $\VVV(S)$ for $S \subset A^\ast$ are the closed sets for a topology on $X$. If $f \in X$ is defined by $f(\fraka) = 1$ for all $\fraka \in I_A$, then $f \notin \VVV(S)$ for any nonempty $S \subset A^\ast$, which means $X$ is irreducible. In particular, $X$ is connected and not Hausdorff ($X$ contains more than one point). If $S \subset A^\ast$ contains no units, then $\VVV(S) \neq \varnothing$ since $f \in \VVV(S)$, where $f(\frakp) = 0$
for all $\frakp \in \Max(A)$. If $S$ contains a unit, then $\VVV(S) = \varnothing$.

For $a \in A^\ast$, define an open set
$$
\DDD(a) = \{f \in X : f((a)) \neq 0\} = X\sm \VVV(\{a\}).
$$
When two spaces are involved, we will sometimes write $\DDD_X(a)$.
If $u \in A^\times$ then $\DDD(u) = X$.

\begin{proposition} For $x, y \in A^\ast$ we have $\DDD(x) \cap \DDD(y) = \DDD(xy)$. The open sets $\DDD(a)$, for $a \in A^\ast$, form a basis for the topology on $X$.
\end{proposition}

\begin{proof} We have $f \in \DDD(xy)$ if and only if $f((x))f((y)) = f((xy)) \neq 0$, if and only if $f((x)) \neq 0$ and $f((y)) \neq 0$, which is equivalent to
$f \in \DDD(x) \cap \DDD(y)$. Suppose $U \subset X$ is open, so $U = X \sm \VVV(S)$ for some $S \subset A^\ast$. Then 
$S = \bigcup_{a \in S}\{a\}$ gives $\VVV(S) = \bigcap_{a \in S}\VVV(\{a\})$ and hence
$U = \bigcup_{a \in S}\DDD(a)$. This shows $\{\DDD(a) : a \in A^\ast\}$ is a basis for the topology on $X$.
\end{proof}

For each subset $Y \subset X$, define a set
$$
\III(Y) = \{a \in A^\ast : \text{$f((a)) = 0$ for all $f \in Y$}\}.
$$
This subset of $A$ is not an ideal, but is closed under $A^\ast$-scaling: $xa \in \III(Y)$ for all $x \in A^\ast$ and $a \in \III(Y)$. The operation $\III$ is inclusion reversing and for any $S \subset A^\ast$, we have $\VVV(\III(\VVV(S))) = \VVV(S)$. The same proof as for affine varieties shows $\VVV(\III(Y)) = \overline{Y}$, the closure of $Y$ in $X$. Additionally, an analogous argument to the affine variety case shows that a closed subset $Y \subset X$ is irreducible if and only if $\III(Y)$ is ``prime" in the sense that if $aa' \in \III(Y)$ then $a \in \III(Y)$ or $a' \in \III(Y)$.

\begin{proposition}
Let $\varphi : A \to B$ be a homomorphism between Dedekind domains and let $X = \MMM(B)$ and $Y = \MMM(A)$. If $a \in A^\ast$ then $(\varphi^\ast)^{-1}(\DDD_Y(a)) = \DDD_X(\varphi(a))$. In particular, $\varphi^\ast : X \to Y$ is continuous.
\end{proposition}

\begin{proof}
We have
$$
g \in (\varphi^\ast)^{-1}(\DDD_Y(a)) \quad \Longleftrightarrow \quad g\big(\varphi_{\text{e}}((a))\big) = 
g(\varphi(a)B) \neq 0 \quad \Longleftrightarrow \quad g \in \DDD_X(\varphi(a)).
$$
Since the open sets $\DDD_Y(a)$ form a basis for the topology on $Y$, this shows $\varphi^\ast$ is continuous.
\end{proof}

\begin{proposition}
Suppose $x \in A^\ast$ and $\{y_i\}_i \subset A^\ast$. If $\DDD(x) \subset \bigcup_i\DDD(y_i)$ then $\DDD(x) \subset \DDD(y_i)$ for some $i$. Each set $\DDD(x)$, and in particular $\MMM(A)$, is quasi-compact.
\end{proposition}

\begin{proof}
Suppose $\DDD(x) \not\subset \DDD(y_i)$ for all $i$. Then for each $i$ there exists $f_i \in \DDD(x)$ such that $f_i \notin \DDD(y_i)$. Define $f \in \MMM(A)$ by 
$$
f(\frakp) = 
\left\{\begin{array}{ll}
0 & \text{if $f_i(\frakp) = 0$ for some $i$} \\
1 & \text{if $f_i(\frakp) \neq 0$ for all $i$}
\end{array} \right.
\vspace{3mm}
$$
for each $\frakp \in \Max(A)$.
Then $f((x)) = 1$ but $f((y_i)) = 0$ for all $i$, so $\DDD(x) \not\subset \bigcup_i\DDD(y_i)$.
Since the sets $\DDD(a)$ for $a \in A^\ast$ form a basis for the topology on $\MMM(A)$, the first part shows $\DDD(x)$ is quasi-compact.
\end{proof}

For any $a \in A$, let $D(a) = \{\frakp \in \Spec(A) : a \notin \frakp\}$ be the usual principal open subset of $\Spec(A)$. The following result relates $\Spec(A)$ and $\MMM(A)$ and will allow us to define a sheaf of rings on $\MMM(A)$.

\begin{proposition}\label{spec}
If $X = \MMM(A)$ and $Z = \Spec(A)$, there is a topological embedding $F : Z \to X$.\end{proposition}

\begin{proof}
Define $F : Z \to X$ by $F(\frakp) = f_\frakp$, where $f_\frakp(\frakp) = 0$ and $f_\frakp(\frakq) = 1$ for all nonzero primes $\frakq \neq \frakp$. Clearly $F$ is injective. For $a \in A^\ast$,
$$
F^{-1}(\DDD(a)) = \{\frakp \in Z : a \notin \frakp\} = D(a),
$$
so $F$ is continuous. Also,
$$
F(D(a)) = \{f_\frakp : a \notin \frakp\} = \im(F) \cap \DDD(a),
$$
which implies $F$ is a homeomorphism onto its image. 
\end{proof}

\begin{definition}
Let $X = \MMM(A)$, $Z = \Spec(A)$, and $F : Z \to X$ the map defined in Proposition \ref{spec}. Define a sheaf of rings on $X$ by $\OOO_X = F_\ast\OOO_Z$.
\end{definition}

By what we saw in the proof of Proposition \ref{spec}, for any $a \in A^\ast$ we have
$$
\OOO_X(\DDD(a)) = \OOO_Z(F^{-1}(\DDD(a))) = \OOO_Z(D(a)) = A_a.
$$
If $x, y \in A^\ast$ and $\DDD(x) \subset \DDD(y)$, then $D(x) \subset D(y)$ and the restriction homomorphism
$$
\OOO_X(\DDD(y)) = A_y \to A_x = \OOO_X(\DDD(x))
$$
is induced by the corresponding map for $\OOO_Z$, and hence is given by
$a/y^n \mapsto ay^{-n}/1$.

Next we consider the structure of the stalk $\OOO_{X, f}$ for $f \in X$. Define $$\ZZZ(f) = \{\frakp \in \Max(A) : f(\frakp) = 0\}.$$
Note that if $f, g \in X$ and $\ZZZ(f) = \ZZZ(g)$, then $\III(\{f\}) = \III(\{g\})$.

\begin{proposition}\label{stalk}
Let $f \in X$. \\
{\upshape (a)} If $\ZZZ(f) = \varnothing$ then $\OOO_{X, f} = \text{\upshape Frac}(A)$. \\
{\upshape (b)} If $\ZZZ(f) = \{\frakp\}$ then $\OOO_{X, f} = A_\frakp$. \\
{\upshape (c)} If $\ZZZ(f)$ is finite then $\OOO_{X, f}$ is a semi-local {\upshape PID}. \\
{\upshape (d)} If $\ZZZ(f) = \Max(A)$ then $\OOO_{X, f} = A$.
\end{proposition}

\begin{proof}
Let $S = \{a \in A^\ast : f((a)) \neq 0\}$, which is a multiplicative subset of $A$. Then
$$
\OOO_{X, f} = \dlim_{\DDD(a) \ni f}\OOO_X(\DDD(a)) = \dlim_{a \in S}A_a = S^{-1}A.
$$
If $\ZZZ(f) = \varnothing$ then $S = A^\ast$, so $\OOO_{X, f} = \text{Frac}(A)$. If $\ZZZ(f) = \{\frakp\}$ then $S = \{a \in A^\ast : \frakp \nmid (a)\} = A \sm \frakp$, and hence $\OOO_{X, f} = A_\frakp$. If $\ZZZ(f) = \{\frakp_1, \ldots, \frakp_r\}$ then 
$$
S = \{a \in A^\ast : \text{$\frakp_i \nmid (a)$ for all $i$}\} = A \sm \bigcup_{i=1}^r\frakp_i.
$$
In this case $\OOO_{X, f} = S^{-1}A$ has maximal ideals $S^{-1}\frakp_1, \ldots, S^{-1}\frakp_r$, and is thus a semi-local Dedekind domain, hence a PID. If $\ZZZ(f) = \Max(A)$ then $S = A^\times$ and hence $\OOO_{X, f} = A$. 
\end{proof}

For any homomorphism $\varphi : A \to B$, we will define a corresponding morphism of ringed spaces $(X, \OOO_X) \to (Y, \OOO_Y)$, where $X = \MMM(B)$ and $Y = \MMM(A)$.
Let $W = \Spec(B)$, $Z = \Spec(A)$, and $(\varphi^{\text{a}}, \varphi^{\natural}) : (W, \OOO_W) \to (Z, \OOO_Z)$ the usual morphism of affine schemes.
On topological spaces, we use $\varphi^\ast : X \to Y$. If $F_A : Z \to Y$ and $F_B : W \to X$ are the maps from Proposition \ref{spec}, the diagram
$$
\xymatrix{
W \ar[rr]^>>>>>>>>>>{\varphi^{\text{a}}} \ar[d]_{F_B} & & Z \ar[d]^{F_A} \\
X \ar[rr]^>>>>>>>>>>{\varphi^\ast} && Y
}
$$
commutes, so there is an induced morphism
$$
\varphi^{\#} = (F_A)_\ast(\varphi^{\natural}) : \OOO_Y = (F_A)_\ast\OOO_Z \to (F_A)_\ast\varphi_\ast^{\text{a}}\OOO_W = (\varphi^\ast)_{\ast}\OOO_X,
$$
giving a morphism of ringed spaces $(\varphi^\ast, \varphi^{\#}) : (X, \OOO_X) \to (Y, \OOO_Y)$. Following through the definitions shows 
$$
\varphi^{\#}(\DDD_Y(a)) : \Gamma(\DDD_Y(a), \OOO_Y) = A_a \to B_{\varphi(a)} = \Gamma(\DDD_Y(a), (\varphi^\ast)_\ast\OOO_X)
$$
is the ring homomorphism induced by $\varphi$: $x/a^n \mapsto \varphi(x)/\varphi(a)^n$.

Note that an element $x/a^n \in \OOO_Y(\DDD_Y(a))$ can be viewed as a function $\alpha : \DDD_Y(a) \to K$ by $$\alpha(f) = \frac{f((x))}{f((a))^n}.$$ If $\DDD_Y(a) \subset \DDD_Y(a')$ then the restriction homomorphism $\OOO_Y(\DDD_Y(a')) \to \OOO_Y(\DDD_Y(a))$ corresponds to restriction of functions $\alpha \mapsto \alpha|_{\DDD_Y(a)}$. Also, if 
$\alpha : \DDD_Y(a) \to K$ is determined by $x/a^n \in \OOO_Y(\DDD_Y(a))$ and
$\beta : \DDD_X(\varphi(a)) \to K$ is determined by the element $\varphi^{\#}(\DDD_Y(a))(x/a^n)$, then $\beta = \alpha \circ \varphi^\ast|_{\DDD_X(\varphi(a))}$.

For any $a \in A$, there is a homeomorphism $\Spec(A_a) \to D(a)$. The following result shows that the analogous map for $\MMM(A)$ is not necessarily surjective.

\begin{proposition}\label{embedding}
Let $a \in A^\ast$. There is a topological embedding $\MMM(A_a) \to \DDD(a)$ with image $$Y = \{f \in X : \text{$f(\frakp) = 1$ for all $\frakp \mid (a)$}\},$$
which defines an isomorphism of ringed spaces $(\MMM(A_a), \OOO_{\MMM(A_a)}) \to
(Y, \OOO_X|_Y)$, where $X = \MMM(A)$.
\end{proposition}

\begin{proof}
Let $X_a = \MMM(A_a)$. Note that $\Spec(A_a) = \{\frakp A_a : \frakp \in \Spec(A), a \notin \frakp\}$. Define $F : X_a \to X$ to be the map corresponding to the inclusion $A \to A_a$, so $F(g)(\frakp) = g(\frakp A_a)$ for any $\frakp \in \Max(A)$. If $f \in X$ satisfies $f(\frakp) = 1$ for all $\frakp \mid (a)$, define $g \in X_a$ by $g(\frakp A_a) = f(\frakp)$ for all $\frakp \in \Max(A)$. This is well-defined by the condition on $f$, and $f = F(g)$. Conversely, suppose $f = F(g)$ for some $g \in X_a$. If $\frakp \mid (a)$ then $\frakp A_a = A_a$ and hence
$f(\frakp) = g((1)) = 1$. This shows $F$ has the stated image $Y$. 

Suppose $g, g' \in X_a$ satisfy $F(g) = F(g')$. Then $g(\frakp A_a) = g'(\frakp A_a)$ for all $\frakp \in \Max(A)$, so $g = g'$ and $F$ is injective. Next, for $x \in A^\ast$,
$$
F^{-1}(\DDD_X(x)) = \{g \in X_a : F(g)((x)) \neq 0\} 
= \{g \in X_a : g(xA_a) \neq 0\} 
= \DDD_{X_a}(x/1),
$$
which shows $F$ is continuous. Now define $G : Y \to X_a$ by $G(f)(\frakp A_a) = f(\frakp)$. Then $G = F^{-1}$ and
for $x/a^n \in A_a^\ast$,
$$
G^{-1}(\DDD_{X_a}(x/a^n)) = G^{-1}(\DDD_{X_a}(x/1)) = \{f \in Y : f((x)) \neq 0\} = \DDD_X(x) \cap Y,
$$
which shows $F$ is a homeomorphism onto its image. 

Since $F = i^\ast$ where $i : A \to A_a$ is the inclusion, there is an induced morphism of sheaves $i^{\#}|_Y : \OOO_X|_Y \to F_\ast\OOO_{X_a}|_Y$. Using Proposition \ref{stalk}, it is easy to check that for any $g \in X_a$, $\OOO_{X_a, g} = \OOO_{X, F(g)}$ as subrings of $\text{Frac}(A)$, so $i^{\#}|_Y$ is an isomorphism.
\end{proof}

In establishing the equivalence between Dedekind domains and certain spaces of totally multiplicative arithmetic functions, we will restrict to ring homomorphisms satisfying the following property.

\begin{definition}
We say a homomorphism of domains $\varphi : R \to S$ is \textit{quasi-integral} if 
for any nonzero prime ideal $\frakq \subset S$, the prime ideal $\varphi^{-1}(\frakq) \subset R$ is nonzero.
\end{definition}

Any integral homomorphism is quasi-integral, but not conversely in general. For example, the inclusion $\ZZ \to \ZZ_{(2)}$ is quasi-integral, but not integral. Also, the property of being quasi-integral is preserved under composition. 

\begin{proposition}\label{zeros}
If $\varphi : A \to B$ is an injective, quasi-integral homomorphism of Dedekind domains, then for each $g \in \MMM(B)$, the map $\ZZZ(g) \to \ZZZ(\varphi^\ast(g))$ given by $\frakq \mapsto \varphi^{-1}(\frakq)$ is surjective.
\end{proposition}

\begin{proof}
First suppose $\ZZZ(g) = \varnothing$ and let $\frakp \in \ZZZ(\varphi^\ast(g))$. Then $\varphi^\ast(g)(\frakp) = g(\varphi(\frakp)B) = 0$, so $\varphi(\frakp)B \neq B$ and $\varphi(\frakp)B \neq (0)$ since $\varphi$ is injective. This implies there is a nonzero prime $\frakq \mid \varphi(\frakp)B$ such that $g(\frakq) = 0$, a contradiction. Therefore, if $\ZZZ(g) = \varnothing$ then $\ZZZ(\varphi^\ast(g)) = \varnothing$. Let $\frakq \in \ZZZ(g)$ and let
$\frakp = \varphi^{-1}(\frakq)$, which is nonzero. Then $\varphi(\frakp)B \subset \frakq$, so $\varphi(\frakp)B \neq B$ and $\frakq \mid \varphi(\frakp)B$. This implies $\varphi^\ast(g)(\frakp) = g(\varphi(\frakp)B) = 0$, which means $\frakp \in \ZZZ(\varphi^\ast(g))$, so contraction gives a map $\ZZZ(g) \to \ZZZ(\varphi^\ast(g))$. To show it is surjective, let $\frakp \in \ZZZ(\varphi^\ast(g))$, so $g(\varphi(\frakp)B) = 0$. As above, there is a nonzero prime $\frakq \mid \varphi(\frakp)B$ such that $g(\frakq) = 0$. Then $\frakp \subset \varphi^{-1}(\frakq)$ and since $\frakp$ is maximal, $\frakp = \varphi^{-1}(\frakq)$.
\end{proof}

Although the ringed space $(\MMM(A), \OOO_{\MMM(A)})$ is not locally ringed, its morphisms satisfy the following condition on stalks which will be our substitute for the usual local ring condition for morphisms of locally ringed spaces.

\begin{lemma}\label{stalk3}
Let $\varphi : A \to B$ be a homomorphism of Dedekind domains, $X = \MMM(B)$, $Y = \MMM(A)$, and $(\varphi^\ast, \varphi^{\#}) : (X, \OOO_X) \to (Y, \OOO_Y)$ the corresponding morphism of ringed spaces. For each $g \in X$, let $\varphi_g^{\#} : \OOO_{Y, \varphi^\ast(g)} \to \OOO_{X, g}$ be the corresponding stalk homomorphism. Then for each maximal ideal $\frakm \subset \OOO_{Y, \varphi^\ast(g)}$, the ideal generated by 
$\varphi^{\#}_g(\frakm)$ in $\OOO_{X, g}$ is proper.
\end{lemma}

\begin{proof}
Fix $g \in X$ and let $f = \varphi^\ast(g)$. By the proof of Proposition 
\ref{stalk}, $\OOO_{X, g} = S_g^{-1}B$ where
$$
S_g = \{b \in B^\ast : g((b)) \neq 0\} = B \sm \bigcup_{\frakq \in \ZZZ(g)}\frakq.
$$
Similarly, $\OOO_{Y, f} = T_f^{-1}A$ where
$$
T_f = \{a \in A^\ast : f((a)) \neq 0\} = A \sm \bigcup_{\frakp \in \ZZZ(f)}\frakp.
$$
There is a commutative diagram
$$
\xymatrix{
A \ar[rr]^>>>>>>>>>>>>>{\varphi} \ar[d]  & & B \ar[d] \\
T_f^{-1}A \ar[rr]^>>>>>>>>>>{\varphi^{\#}_g} && S_g^{-1}B
}
$$
where $\varphi^{\#}_g$ is the obvious map induced by $\varphi$. Suppose $\frakm \subset T_f^{-1}A$ is a maximal ideal. Assuming $T_f^{-1}A \neq \text{Frac}(A)$, $\frakm = T_f^{-1}\frakp$ for some nonzero prime ideal $\frakp \subset A$ satisfying $\frakp \cap T_f = \varnothing$. From the definition of $T_f$, it follows that 
$\frakp \in \ZZZ(f)$. Then by Proposition \ref{zeros}, 
$$
\ZZZ(f) = \ZZZ(\varphi^{\ast}(g)) = \{\varphi^{-1}(\frakq) : \frakq \in \ZZZ(g)\},
$$
which implies $\varphi(\frakp)B \subset \frakq$ for some $\frakq \in \ZZZ(g)$. Note that $\ZZZ(g) \neq \varnothing$ since $T_f^{-1}A \neq \text{Frac}(A)$. Hence
$\varphi(\frakp)B \cap S_g = \varnothing$, so the ideal in $S_g^{-1}B$ generated by
$\varphi^{\#}_g(\frakm)$ is proper. This is also clearly true if $T_f^{-1}A = \text{Frac}(A)$.
\end{proof}

Our next goal is to determine the dimension of the space $\MMM(A)$. To do this, we use the following Nullstellensatz-type result.

\begin{lemma}\label{lemma}\label{null}
Let $A$ be a Dedekind domain and $\frakp_1, \ldots, \frakp_m \in \Max(A)$, and let
$\frakp_i^\ast = \frakp_i\sm\{0\}$. If $P_m = \bigcup_{i=1}^m\frakp_i^\ast$ then $\III(\VVV(P_m)) = P_m$.
\end{lemma}

\begin{proof}
Obviously $P_m \subset \III(\VVV(P_m))$, so we need to prove the opposite inclusion. Define $f \in \MMM(A)$ by $f(\frakp_i) = 0$ for all $1 \leq i \leq m$ and $f(\frakq) = 1$ for all $\frakq \notin \{\frakp_1, \ldots, \frakp_m\}$.
Then 
$$
\III(\{f\}) = \{a \in A^\ast : \text{$\frakp_i \mid (a)$ for some $1 \leq i \leq m$}\} = P_m.
$$
Suppose $\VVV(S) \subset \MMM(A)$ is a closed set containing $f$. Then for all $a \in S$, 
there is a $\frakp_i \mid (a)$, which implies $\VVV(P_m) \subset \VVV(S)$. Since $f \in \VVV(P_m)$, this shows $\overline{\{f\}} = 
\VVV(P_m)$ and hence
\begin{equation*}
\III(\VVV(P_m)) = \III(\overline{\{f\}}) \subset   \III(\{f\}) = P_m. \qedhere
\end{equation*}
\end{proof}

As $\VVV(P_m)$ is the closure of a one-point set, it is an irreducible subset of $\MMM(A)$. Also, since $P_{m-1} \neq P_m$, we have $\VVV(P_{m-1}) \neq \VVV(P_m)$.

\begin{proposition}\label{dimension}
Let $A$ be a Dedekind domain and $X = \MMM(A)$. Then $\dim X = \#\Max(A)$, where the dimension is $\infty$ if $\Max(A)$ is infinite.
\end{proposition}

\begin{proof}
First suppose $\Max(A)$ is infinite and let $\{\frakp_i : i \geq 1\} \subset \Max(A)$ be a countable subset. For $m \geq 1$, define $P_m = \bigcup_{i=1}^m\frakp_i^\ast$ and $Z_m = \VVV(P_m)$, which is a closed, irreducible subset of $X$. Then there is an infinite chain
$$
Z_1 \supsetneq Z_2 \supsetneq Z_3 \supsetneq \cdots
$$
which shows $\dim X = \infty$.

Now suppose $\Max(A) = \{\frakp_1, \ldots, \frakp_n\}$. As above there is a chain of closed, irreducible subsets
\begin{equation}\label{chain}
X \supsetneq Z_1 \supsetneq Z_2 \supsetneq \cdots \supsetneq Z_n.
\end{equation}
Note that $Z_1 \neq X$ since the function that is identically 1 is not in $Z_1$. Suppose $\VVV(S) \subset X$ is a closed irreducible subset and $Z_m \subset \VVV(S) \subset Z_{m-1}$ for some $m$. Then by Lemma \ref{lemma}, $P_m \supset \III(\VVV(S)) \supset P_{m-1}$.
Since $A$ is a semi-local Dedekind domain, it is a PID, and thus $\frakp_i^\ast = A^\ast p_i$
for some prime element $p_i \in A$. Suppose $\III(\VVV(S)) \neq P_{m-1}$. Then there exists $a \in \III(\VVV(S))$ such that $a \notin P_{m-1}$, and hence $a \in \frakp_m^\ast$. Write $a = p_mx_1$ for some $x_1 \in A^\ast$. Then $a = p_mx_1 \in \III(\VVV(S))$, and $\VVV(S)$ is irreducible, so $p_m \in \III(\VVV(S))$ or $x_1 \in \III(\VVV(S))$. If $p_m \in \III(\VVV(S))$ then $\frakp_m^\ast = 
A^\ast p_m \subset \III(\VVV(S))$, which means $\III(\VVV(S)) = P_m$.
If $p_m \notin \III(\VVV(S))$ then $x_1 \in \III(\VVV(S))$, and since $a \notin P_{m-1}$, we have $x_1 \notin P_{m-1}$. It follows that
$x_1 \in \frakp_m^\ast$ and $x_1 = p_mx_2$ for some $x_2 \in A^\ast$, so $a = p_m^2x_2$. Again, since $p_m \notin \III(\VVV(S))$ and $\VVV(S)$ is irreducible, $a = p_m^3x_3$ for some $x_3 \in A^\ast$. By induction, we find $p_m^k \mid a$ for all $k \geq 1$, which means $a = 0$, a contradiction. This shows we must have $\VVV(S) = Z_{m-1}$ or $\VVV(S) = Z_m$.

Next suppose $X \supset \VVV(S) \supset Z_1$ with $\VVV(S)$ closed and irreducible, so 
$\III(\VVV(S)) \subset \frakp_1^\ast$. Let $a \in \III(\VVV(S))$, so $a = p_1x_1$ for some
$x_1 \in A^\ast$. Since $\VVV(S)$ is irreducible, $p_1 \in \III(\VVV(S))$, in which case $\III(\VVV(S)) = \frakp_1^\ast$, or $x_1 \in \III(\VVV(S))$. If $p_1 \notin \III(\VVV(S))$, using the same argument as in the last paragraph, we conclude $a = 0$, a contradiction, so if $\III(\VVV(S)) \neq \frakp_1^\ast$ then
$\III(\VVV(S)) = \varnothing$. This shows $\VVV(S) = Z_1$ or $\VVV(S) = X$. Finally, if
$Z_n \supsetneq \VVV(S)$ with $\VVV(S)$ closed and irreducible, then $P_n \subsetneq \III(\VVV(S))$ and $\III(\VVV(S))$ must contain a unit since 
$A^\times = A\sm\bigcup_{i=1}^n\frakp_i$. Then $\III(\VVV(S)) = A^\ast$ and hence $\VVV(S) = \varnothing$. Therefore the chain (\ref{chain}) is maximal and $\dim X = n$.
\end{proof}

\begin{proposition}\label{noetherian}
The topological space $X = \MMM(A)$ is Noetherian if and only if $\Max(A)$ is finite.
\end{proposition}

\begin{proof}
If $\Max(A)$ is infinite, the first paragraph of the previous proof shows $X$ is not Noetherian.
Now suppose $\Max(A) = \{\frakp_1, \ldots, \frakp_n\}$. Note that for $S \subset A^\ast$ and 
$f, g \in X$, if $\ZZZ(f) = \ZZZ(g)$ then $f \in \VVV(S)$ if and only if $g \in \VVV(S)$. For any
integers $1 \leq i_1 < \cdots < i_k \leq n$, define $f_{i_1\ldots i_k} \in X$ by
$$
f_{i_1\ldots i_k}(\frakp_i) = 
\left\{\begin{array}{ll}
0 & \text{if $i \in \{i_1, \ldots, i_k\}$} \\
1 & \text{otherwise}.
\end{array} \right.
$$
Writing $\mathscr{P}(\cdot)$ for power set, define
$$
F : \{\VVV(S) \neq \varnothing : S \subset A^\ast\} \to \mathscr{P}(\{f_{i_1\ldots i_k} : 
1 \leq i_1 < \cdots < i_k \leq n\})
$$
by $F(\VVV(S)) = \{f_{i_1\ldots i_k} \in \VVV(S) : 1 \leq i_1 < \cdots < i_k \leq n\}$. To show $F$ is injective, suppose
$$
\{f_{i_1\ldots i_k} \in \VVV(S) : 1 \leq i_1 < \cdots < i_k \leq n\} = \{f_{i_1\ldots i_\ell} \in \VVV(S') : 1 \leq i_1 < \cdots < i_\ell \leq n\}
$$
and let $f \in \VVV(S)$. Then $\ZZZ(f) = \{\frakp_{i_1}, \ldots, \frakp_{i_k}\}$ for some
$i_1 < \cdots < i_k$. As $\ZZZ(f) = \ZZZ(f_{i_1\ldots i_k})$, we have $f \in \VVV(S')$. 
The same argument gives the reverse inclusion, so $\VVV(S) = \VVV(S')$. Since
$F$ is injective, $X$ has only finitely many closed sets and therefore is Noetherian.
\end{proof}

\begin{remark}
Recall that a topological space $X$ is called $T_0$ if, among other equivalent definitions, every irreducible closed subset of $X$ has at most one generic point. Let $X = \MMM(A)$, $\frakp \in \Max(A)$, and set $\frakp^\ast = \frakp\sm\{0\}$. Then by Lemma \ref{null} and its proof, $Y = \VVV(\frakp^\ast)$ is an irreducible, closed subset of $X$ and $Y = \overline{\{f\}}$ for any $f \in X$ with $\ZZZ(f) = \{\frakp\}$. It follows that if $\#\Max(A) > 1$ and $\#K > 2$ then $Y$ has more than one generic point and therefore $X$ is not $T_0$. 
\end{remark}

\begin{proposition}\label{points}
Any $f \in \MMM(A)$ not equal to the identity function $e \in \MMM(A)$ is not a closed point.
The point $e$ is closed if and only if the ideal class group $\Cl(A)$ is torsion.
\end{proposition}

\begin{proof}
Since $e(\fraka) = 0$ for all $\fraka \neq (1)$, the point $e$ lies in every nonempty closed subset of $\MMM(A)$, which proves the first statement. For the second statement, first suppose $\Cl(A)$ is a torsion group and let $f \in \VVV(\III(\{e\}))$. Then $f((a)) = 0$ for all
$a \in A^\ast\sm A^\times$. Given any $\fraka \in I_A$, the class $[\fraka] \in \Cl(A)$ has finite order, which means $\fraka^n = (a)$ for some $n \geq 1$ and $a \in A^\ast$. If $a \in A^\times$ then $\fraka = (1)$ and $f(\fraka) = 1$. If $a \in A^\ast\sm A^\times$ then
$f(\fraka) = 0$, which shows $f = e$ and hence $\overline{\{e\}} = \{e\}$. 

Conversely, suppose $e$ is closed and suppose $\Cl(A)$ contains an element with infinite order. Represent this class by an integral ideal $\fraka$. Writing $\fraka = \frakp_1^{n_1}\cdots\frakp_r^{n_r}$, at least one class $[\frakp_k] \in \Cl(A)$ has infinite order, so no power of $\frakp_k$ is principal. Define $f \in \MMM(A)$ by $f(\frakp_k) = 1$ and $f(\frakq) = 0$ for all primes $\frakq \neq \frakp_k$. Any $a \in A^\ast\sm A^\times$ factors into primes as $(a) = \frakq_1\cdots\frakq_s$, where at least one of the $\frakq_i$ is not equal to $\frakp_k$ since no power of $\frakp_k$ is principal. Then
$f((a)) = 0$, which means $f \in \VVV(\III(\{e\})) = \overline{\{e\}} = \{e\}$, a contradiction.
\end{proof}

\section{A refined space}

Let $\mathbf{A}$ be the category whose objects are Dedekind domains and whose morphisms are injective, quasi-integral ring homomorphisms. Let $\mathbf{B}$ be the category whose objects are ringed spaces $(X, \OOO_X)$ and whose morphisms are morphisms of ringed spaces $(\Phi, \Phi^{\#}) : (X, \OOO_X) \to (Y,  \OOO_Y)$ such that $\Phi^{\#}(Y) : \OOO_{Y}(Y) \to
\OOO_{X}(X)$ is injective and quasi-integral, and for each $x \in X$, the stalk map $\Phi^{\#}_x : \OOO_{Y, \Phi(x)} \to \OOO_{X, x}$ satisfies the following:
for each maximal ideal $\frakm \subset \OOO_{Y, \Phi(x)}$, the ideal $\Phi^{\#}_x(\frakm)\OOO_{X, x}$ is proper. 

Note that $\MMM$ defines a functor from $\mathbf{A}$ to $\mathbf{B}$ by Lemma \ref{stalk3}. For any objects $A$ and $B$ of $\mathbf{A}$, there is a map
$$
\Gamma : \Hom_{\mathbf{B}}(\MMM(B), \MMM(A)) \to \Hom_{\mathbf{A}}(A, B)
$$
given by $\Gamma(\Phi, \Phi^{\#}) = \Phi^{\#}(\MMM(A))$, the homomorphism induced on global sections. By construction, $\Gamma \circ \MMM = \id_{\Hom_{\mathbf{A}}(A, B)}$, so $\MMM$ is injective on morphisms. This shows $\MMM$ is faithful, but the following simple example shows it is not full.

\begin{example} 
Suppose $K$ does not have characteristic $2$.
Let $X = \MMM(\ZZ)$ and define $(\Phi, \Phi^\#) : (X, \OOO_X) \to (X, \OOO_X)$ in the following way. On points, $\Phi(f) = f^2$, where $f^2(\fraka) = f(\fraka)^2$. For $a \in \ZZ$ nonzero,
$$
\Phi^{-1}(\DDD(a)) = \{f \in X : f^2((a)) \neq 0\} = \DDD(a),
$$
so $\Phi$ is continuous. Define 
$$
\Phi^{\#}(\DDD(a)) : \OOO_X(\DDD(a)) = \ZZ[1/a] \to \ZZ[1/a] = \OOO_X(\Phi^{-1}(\DDD(a)))
$$
to be the identity for each $a$. For each $f \in X$, the functions $f$ and $f^2$ have the same zeros, so $\OOO_{X, f^2} = \OOO_{X, f}$ and $\Phi^{\#}_f : \OOO_{X, f^2} \to \OOO_{X, f}$ is the identity. Hence $(\Phi, \Phi^{\#})$ is a nonidentity element of $\Hom_{\mathbf{B}}(X, X)$, but $\Hom_{\mathbf{A}}(\ZZ, \ZZ) = \{\id_\ZZ\}$, so $\MMM : \Hom_{\mathbf{A}}(\ZZ, \ZZ)
\to \Hom_{\mathbf{B}}(X, X)$ is not surjective.
\end{example}

To define a space that gives a fully faithful functor, we modify $\MMM(A)$ in the following way.
Define an equivalence relation $\sim$ on $X = \MMM(A)$ by $f \sim g$ if $\ZZZ(f) = \ZZZ(g)$. 
We will write $[f]$ for the equivalence class of $f \in X$ and $\ZZZ([f])$ for the set $\ZZZ(g)$ for any $g \in [f]$.
Let $\MMM_1(A) = X/\hspace{-1.2mm}\sim$ and let $\pi : X \to \MMM_1(A)$ be the projection $\pi(f) = [f]$. Give $\MMM_1(A)$ the quotient topology determined by the topology on $X$. Let $Y = \MMM_1(A)$ and define a sheaf of rings $\OOO_Y = \pi_\ast\OOO_X$ on $Y$. For $a \in A^\ast$, set $\DDD_1(a) = \pi(\DDD(a))$, and for $S \subset A^\ast$, set
$\VVV_1(S) = \pi(\VVV(S))$. Suppose
$f \in \pi^{-1}(\DDD_1(a))$. Then $\ZZZ(f) = \ZZZ(g)$ for some $g \in \DDD(a)$, which implies
$f \in \DDD(a)$. Therefore $\pi^{-1}(\DDD_1(a)) = \DDD(a)$, which means $\pi$ is an open map, $\{\DDD_1(a) : a \in A^\ast\}$ is a basis for the topology on $Y$, and
$$
\OOO_Y(\DDD_1(a)) = \OOO_X(\DDD(a)) = A_a.
$$
This also implies that for any $f \in X$,
\begin{equation}\label{stalk2}
\OOO_{Y, [f]} = \dlim_{\DDD_1(a) \ni [f]} (\pi_\ast\OOO_X)(\DDD_1(a))
= \dlim_{\DDD(a) \ni f} \OOO_X(\DDD(a)) = \OOO_{X, f}.
\end{equation}
Similarly, if $f \in \pi^{-1}(\VVV_1(S))$ then $\ZZZ(f) = 
\ZZZ(g)$ for some $g \in \VVV(S)$, which implies $f \in \VVV(S)$ and thus $\pi^{-1}(\VVV_1(S)) = \VVV(S)$. This shows $\pi$ is also a closed map and the closed sets of $Y$ are $\VVV_1(S)$ for $S \subset A^\ast$. If $e \in X$ is the identity function, then
$$
[e] = \{f \in X : \ZZZ(f) = \Max(A)\} = \{e\},
$$
so from $\pi$ being a closed map, the results of Proposition \ref{points} continue to hold for $[e] \in Y$.

\begin{corollary}
The embedding $\MMM(A_a) \to \DDD(a)$ in Proposition {\upshape \ref{embedding}} induces an isomorphism of ringed spaces $\MMM_1(A_a) \to \DDD_1(a)$. The embedding $\Spec(A) \to \MMM(A)$ in 
Proposition {\upshape \ref{spec}} induces a homeomorphism $$\Spec(A) \to \{y \in Y : \#\ZZZ(y) \leq 1\},\vspace{2mm}$$ where the prime ideal $(0)$ corresponds to the point $y \in Y$ with $\ZZZ(y) = \varnothing$. This point $y \in Y$ lies in every basis open set $\DDD_1(a)$, so $y$ is a generic point of $Y$, and is the unique such point.
\end{corollary}

\begin{proposition}\label{dim2}
If $A$ is a Dedekind domain and $Y = \MMM_1(A)$, then $\dim Y = \#\Max(A)$. If $\Max(A)$ is finite, then $Y$ is finite and contains $2^{\#\Max(A)}$ points.
\end{proposition}

\begin{proof}
Let $Z \subset Y$ and $\pi : X \to Y$ be the projection. If $\pi^{-1}(Z) \subset X$ is closed and irreducible, then $\pi(\pi^{-1}(Z)) = Z$ is closed and irreducible since $\pi$ is continuous and closed.
Conversely, suppose $Z = \VVV_1(S)$ is closed and irreducible, so $\pi^{-1}(Z) = \VVV(S)$. Suppose $\VVV(S) = W_1 \cup W_2$ with $W_1, W_2 \subsetneq \VVV(S)$ closed. Then $W_i = \VVV(S_i)$ for some $S_i$ and $Z = \pi(\pi^{-1}(Z)) = \pi(W_1) \cup \pi(W_2)$. If $\pi(W_i) = Z$
then 
$$
W_i = \VVV(S_i) = \pi^{-1}(\pi(\VVV(S_i))) = \pi^{-1}(\VVV_1(S)) = \VVV(S),
$$
a contradiction. Therefore $\pi(W_1), \pi(W_2) \subsetneq Z$ and each set is closed in $Z$, contradicting $Z$ being irreducible. Hence $\pi^{-1}(Z)$ is closed and irreducible. The first statement now follows from
Proposition \ref{dimension}.

Next, suppose $\#\Max(A) = n$. Each $y \in Y$ is determined by its prime zeros, so choosing a $y$ is equivalent to choosing a subset of $\Max(A)$. Therefore $\#Y = 2^n$.
\end{proof}

\begin{corollary}
The topological space $\MMM_1(A)$ is Noetherian if and only if $\Max(A)$ is finite.
\end{corollary}

\begin{proof}
Since there is a bijection between the closed subsets $W \subset \MMM(A)$ and the closed subsets $Z \subset \MMM_1(A)$ given by $W \mapsto \pi(W)$ and $Z \mapsto \pi^{-1}(Z)$, the result follows from Proposition \ref{noetherian}.
\end{proof}

Recall that a topological space $X$ is called \textit{spectral} if $X$ is quasi-compact, every irreducible, closed subset of $X$ has a unique generic point, and $X$ has a basis consisting of quasi-compact open subsets which is closed under finite intersections.

\begin{lemma}
If $a, a' \in A^\ast$ then $\DDD_1(a) \cap \DDD_1(a') = \DDD_1(aa')$.
\end{lemma}

\begin{proof}
Since $\DDD(a) \cap \DDD(a') = \DDD(aa')$, we have $\DDD_1(aa') \subset \DDD_1(a) \cap \DDD_1(a')$. Now let $x \in \DDD_1(a) \cap \DDD_1(a')$, so $x = \pi(f) = \pi(f')$ for some $f \in \DDD(a)$ and $f' \in \DDD(a')$. Then $f \sim f'$, which means $\ZZZ(f) = \ZZZ(f')$. Hence $f' \in \DDD(a)$, so $f' \in \DDD(a) \cap \DDD(a') = \DDD(aa')$ and therefore $x = \pi(f') \in \DDD_1(aa')$.
\end{proof}

\begin{proposition}\label{basis}
Let $X = \MMM(A)$ and $Y = \MMM_1(A)$.\\
{\upshape (a)} The space $Y$ is $T_0$ if and only if the ideal class group $\Cl(A)$ is torsion. \\
{\upshape (b)} If $A$ is a {\upshape PID} then $Y$ is a spectral space.
\end{proposition}

\begin{proof}
(a) Suppose $\Cl(A)$ is torsion and $[f] \neq [g]$ in $Y$. By interchanging $f$ and $g$, we may assume $\ZZZ(f) \neq \varnothing$, so there exists $\frakp \in \Max(A)$ such that 
$f(\frakp) = 0$ but $g(\frakp) \neq 0$. Since $\Cl(A)$ is torsion, $\frakp^n = (a)$ for some integer $n \geq 1$ and $a \in A^\ast$. Then $[f] \notin \DDD_1(a)$ but $[g] \in \DDD_1(a)$, which shows $Y$ is $T_0$. Now suppose $\Cl(A)$ is not torsion. As in the proof of Proposition \ref{points}, there is an
$f \in X$ such that $\ZZZ(f) \neq \ZZZ(e)$ and $f \in \overline{\{e\}}$. Then since $\pi$ is a closed map, any closed subset of $Y$ containing $[e]$ also contains $[f]$. As $[f] \neq [e]$, this shows $Y$ is not $T_0$.

(b) Suppose $A$ is a PID. If $S \subset A^\ast$ and $f \in X$, note that the closure of $\{f\}$ in $X$ is equal to
$\VVV(\III(\{f\})) = \VVV(S(f))$, where
$$
S(f) = \{a \in A^\ast : f((a)) = 0\} = \bigcup_{\frakp \in \ZZZ(f)} \frakp^\ast.
$$
Suppose $W = \VVV(S)$ is an irreducible, closed subset of $X$. We will show $W$ contains a generic point. Since $W$ is irreducible, using Proposition \ref{properties}, it follows that every element of $S$ is prime. Define $f_0 \in X$ by $f_0((p)) = 0$ for all $p \in S$ and $f_0((q)) = 1$ for all primes $q \notin S$, so $f_0 \in W$. Then $\overline{\{f_0\}} = \VVV(S(f_0))$, where $S(f_0) = \bigcup_{p \in S} A^\ast p$.
If $f \in W$ then $f((p)) = 0$ for all $p \in S$, so $f((a)) = 0$ for all $a \in S(f_0)$. Hence
$\overline{\{f_0\}} = W$ and $f_0$ is a generic point of $W$. 

Now, if $Z \subset Y$ is an irreducible, closed subset, then $\pi^{-1}(Z) \subset X$ is irreducible and closed by the proof of Proposition \ref{dim2}. If $\eta \in \pi^{-1}(Z)$ is a generic point, then $\pi(\eta) \in Z$ is a generic point since $\pi$ is closed. By part (a), this generic point is unique. We have seen that $\{\DDD_1(a) : a \in A^\ast\}$ is a basis of open sets in $Y$, closed under finite intersections by Lemma \ref{basis}, and each $\DDD_1(a) \cong \MMM_1(A_a)$ is quasi-compact. As $Y$ is quasi-compact, it is a spectral space.
\end{proof}

\begin{example}
Let $A = \ZZ[\sqrt{-5}]$, so $\#\Cl(A) = 2$, and let $X = \MMM(A)$. In $A$, factor $(3) = \frakp_+\frakp_{-}$ where $\frakp_{\pm} = (3, 1 \pm \sqrt{-5})$. Let $Z = \VVV(\{3\}) \subset X$, which is irreducible since $3$ is an irreducible element of $A$, and define $f \in X$ by $f(\frakp_+) = 0$ and $f(\frakq) = 1$ for all $\frakq \neq \frakp_+$, so $f \in Z$ and $\overline{\{f\}} = \VVV(\frakp_+^\ast)$. Similarly define $g \in X$ by $g(\frakp_{-}) = 0$ and $g(\frakq) = 1$ for all $
\frakq \neq \frakp_{-}$, so again $g \in Z$. However, $g \notin \VVV(\frakp_+^\ast)$ since, for example, $\frakp_2\frakp_+ = (1 + \sqrt{-5})$ with $\frakp_2 = (2, 1 + \sqrt{-5})$ prime and
hence $g((1 + \sqrt{-5})) \neq 0$. Thus $\overline{\{f\}} \subsetneq Z$, and similarly $\overline{\{g\}} \subsetneq Z$. If $h \in Z$ is arbitrary, then $\ZZZ(h)$ contains at least one of $\frakp_{\pm}$, and if $\frakp_+ \in \ZZZ(h)$, then $\{a \in A^\ast : f((a)) = 0\} \subset \{a \in A^\ast : h((a)) = 0\}$ and thus $\overline{\{h\}} \subset \overline{\{f\}} \subsetneq Z$. The same holds if
$\frakp_{-} \in \ZZZ(h)$, which shows $Z$ has no generic point. This also implies $\pi(Z) \subset \MMM_1(A)$ has no generic point, and in particular, $\MMM_1(A)$ is not spectral.
\end{example}

\begin{example}
Let $Y = \MMM_1(A)$ and suppose $\dim Y = 1$. Then $A$ is a discrete valuation ring and
$Y = \{[e], [f]\}$, where $e$ is the usual identity function and $f(\frakp) = 1$, where $\Spec(A) = \{(0), \frakp\}$. Let $Z = \Spec(A)$. There is an isomorphism of 
ringed spaces $F : (Z, \OOO_Z) \to (Y, \OOO_Y)$ given on points by $F(\frakp) = [e]$ and
$F((0)) = [f]$. The morphism of sheaves $\OOO_Y \to F_\ast\OOO_Z$ is the identity by definition.
\end{example}

Suppose $\varphi : A \to B$ is a homomorphism of Dedekind domains that is injective and quasi-integral. Let $X = \MMM(A)$, $X' = \MMM(B)$, $Y = \MMM_1(A)$, $Y' = \MMM_1(B)$, and let $\pi : X \to Y$ and $\pi' : X' \to Y'$ be the projections. Suppose $f, g \in X'$ and $f \sim g$. Then by Proposition \ref{zeros},
$$
\ZZZ(\varphi^\ast(f)) = \{\varphi^{-1}(\frakq) : \frakq \in \ZZZ(f)\} = \{\varphi^{-1}(\frakq) : \frakq \in \ZZZ(g)\} = \ZZZ(\varphi^\ast(g)),
$$
which means $\varphi^\ast(f) \sim \varphi^\ast(g)$. It follows that there is a continuous map
$\widetilde{\varphi}^\ast : Y' \to Y$ such that the diagram
$$
\xymatrix{
X' \ar[rr]^>>>>>>>>>>{\varphi^{\ast}} \ar[d]_{\pi'}  & & X \ar[d]^{\pi} \\
Y' \ar[rr]^>>>>>>>>>>{\widetilde{\varphi}^{\ast}} && Y
}
$$
commutes, that is, $\widetilde{\varphi}^{\ast}[f] = [\varphi^\ast(f)]$. Define a morphism of sheaves 
$$
\widetilde{\varphi}^{\#} = \pi_\ast(\varphi^{\#}) : \OOO_Y = \pi_\ast\OOO_X \to \pi_\ast(\varphi^\ast)_\ast\OOO_{X'} = (\widetilde{\varphi}^\ast)_\ast\OOO_{Y'}.
$$
Then
$$
\widetilde{\varphi}^{\#}(\DDD_{X, 1}(a)) : \Gamma(\DDD_{X, 1}(a), \OOO_Y) = A_a \to B_{\varphi(a)} = \Gamma(\DDD_{X, 1}(a), (\widetilde{\varphi}^\ast)_\ast\OOO_{Y'})
$$
is the homomorphism induced by $\varphi$. This defines a morphism
$(\widetilde{\varphi}^\ast, \widetilde{\varphi}^{\#}) : (Y', \OOO_{Y'}) \to (Y, \OOO_Y)$ of ringed spaces. 

Fix
$y' \in Y'$ and let $y = \widetilde{\varphi}^{\ast}(y')$. By (\ref{stalk2}) we have $\OOO_{Y', y'} = S_{y'}^{-1}B$, where
$$
S_{y'} = \{b \in B^\ast : \text{$f((b)) \neq 0$ for any $f \in y'$}\} = B \sm \bigcup_{\frakq \in \ZZZ(y')}\frakq.
$$
Similarly, $\OOO_{Y, y} = T_y^{-1}A$ where
$$
T_y = \{a \in A^\ast : \text{$g((a)) \neq 0$ for any $g \in y$}\} = A \sm \bigcup_{\frakp \in \ZZZ(y)}\frakp.
$$
Just as in Lemma \ref{stalk3}, there is a commutative diagram
$$
\xymatrix{
A \ar[rr]^>>>>>>>>>>>>>{\varphi} \ar[d]  & & B \ar[d] \\
T_y^{-1}A \ar[rr]^>>>>>>>>>>{\widetilde{\varphi}^{\#}_{y'}} && S_{y'}^{-1}B
}
$$
and the ideal in $S_{y'}^{-1}B$ generated by $\widetilde{\varphi}^{\#}_{y'}(\frakm)$ is proper for every maximal ideal $\frakm \subset T_y^{-1}A$.

We now come to our main result. Let $\mathbf{A}$ and $\mathbf{B}$ be the categories introduced at the beginning of this section. By what we just showed, $\MMM_1$ is a functor from 
$\mathbf{A}$ to $\mathbf{B}$.

\begin{theorem}\label{equivalence}
The functor $\MMM_1 : \mathbf{A} \to \mathbf{B}$ is fully faithful.
\end{theorem}

\begin{proof}
We need to show that for any objects $A$ and $B$ in $\mathbf{A}$, the map
$$
\MMM_1 : \Hom_{\mathbf{A}}(A, B) \to \Hom_{\mathbf{B}}(\MMM_1(B), \MMM_1(A)),
$$
given by $\MMM_1(\varphi) = (\widetilde{\varphi}^\ast, \widetilde{\varphi}^{\#})$, is a bijection.
Define 
$$
\Gamma: \Hom_{\mathbf{B}}(\MMM_1(B), \MMM_1(A)) \to \Hom_{\mathbf{A}}(A, B) 
$$
by $\Gamma(\Phi, \Phi^{\#}) = \Phi^{\#}(\MMM_1(A))$. If $\varphi : A \to B$ is a morphism in $\mathbf{A}$, then by construction, the homomorphism 
$$
\Gamma(\MMM_1(\varphi)) = \widetilde{\varphi}^{\#}(\MMM_1(A)) : \OOO_{\MMM_1(A)}(\MMM_1(A)) = A \to B = \OOO_{\MMM_1(B)}(\MMM_1(B))
$$
is equal to $\varphi$.

Now let $Y = \MMM_1(A)$ and $X = \MMM_1(B)$, and suppose $(\Phi, \Phi^{\#}) : (X, \OOO_X) \to 
(Y, \OOO_Y)$ is a morphism in $\mathbf{B}$. Let $\varphi = \Gamma(\Phi, \Phi^{\#}) : A \to B$.
We need to show $\MMM_1(\varphi) = (\Phi, \Phi^{\#})$. To show $\widetilde{\varphi}^{\ast} = \Phi$, it is enough to show $\ZZZ(\widetilde{\varphi}^{\ast}(x)) = \ZZZ(\Phi(x))$ for all $x \in X$.
Fix an $x \in X$ and let $y = \Phi(x)$. As above, $\OOO_{X, x} = S_x^{-1}B$, where 
$$
S_x = \{b \in B^\ast : \text{$f((b)) \neq 0$ for any $f \in x$}\} = B\sm\bigcup_{\frakq \in \ZZZ(x)}\frakq
$$
and $\OOO_{Y, y} = T_y^{-1}A$, where 
$$
T_y = \{a \in A^\ast : \text{$g((a)) \neq 0$ for any $g \in y$}\} = A \sm\bigcup_{\frakp \in \ZZZ(y)}\frakp.
$$ 
Also, there is a commutative diagram
\begin{equation}\label{diagram2}
\xymatrix{
A \ar[rr]^>>>>>>>>>>>>>{\varphi} \ar[d]_{\alpha}  & & B \ar[d]^{\beta} \\
T_y^{-1}A \ar[rr]^>>>>>>>>>>{\Phi^{\#}_x} && S_x^{-1}B
}
\end{equation}
where $\alpha$ and $\beta$ are the natural inclusions, and $\Phi^{\#}_x : \OOO_{Y, y} \to \OOO_{X, x}$ is the map on stalks induced by $\Phi^{\#}$. Since $\varphi$ is injective, 
so is $\Phi_x^{\#}$. Also, since $\varphi$ is quasi-integral, so is $\Phi_x^{\#}$.
To see this, suppose $S_x^{-1}B \neq \text{Frac}(B)$ and let $\frakQ \subset S_x^{-1}B$ be a nonzero prime ideal. Then $\frakQ = 
S_x^{-1}\frakq$ for some nonzero prime $\frakq \subset B$ and hence $\frakp = \varphi^{-1}(\frakq) \subset A$ is nonzero. Suppose $(\Phi_x^{\#})^{-1}(\frakQ) = (0)$. 
Then since $\alpha$ is injective, 
$$
\frakp = \varphi^{-1}(\beta^{-1}(\frakQ)) = \alpha^{-1}((\Phi_x^{\#})^{-1}(\frakQ)) = \alpha^{-1}(0)
= (0),
$$
a contradiction. If $S_x^{-1}B = \text{Frac}(B)$, the quasi-integral property is vacuously satisfied.

To show $\ZZZ(y) = \ZZZ(\widetilde{\varphi}^{\ast}(x))$, first suppose $\ZZZ(y) = \varnothing$, so $\OOO_{Y, y} = \text{Frac}(A)$. If $S_x^{-1}B \neq \text{Frac}(B)$ then since $\Phi^{\#}_x$ is quasi-integral, if $\frakQ \subset S_x^{-1}B$ is a nonzero prime, then $(\Phi^{\#}_x)^{-1}(\frakQ)$ is a nonzero prime in $\OOO_{Y, y}$, a contradiction. Hence $\OOO_{X, x} =
\text{Frac}(B)$, which implies $\ZZZ(x) = \varnothing$ and thus $\ZZZ(\widetilde{\varphi}^\ast(x)) = \varnothing$. Conversely, if $\ZZZ(\widetilde{\varphi}^\ast(x)) = \varnothing$ then
$\ZZZ(x) = \varnothing$ and $\OOO_{X, x} = \text{Frac}(B)$. Since $\Phi_x^{\#}$ is injective and by what we assumed about stalk maps,
if $\mathfrak{A} \subset T_y^{-1}A$ is a nonzero proper ideal, then the ideal in $S_x^{-1}B = \text{Frac}(B)$ generated by $\Phi^{\#}_x(\mathfrak{A})$ is nonzero and proper, a contradiction. 
Therefore $\OOO_{Y, y} = \text{Frac}(A)$ and $\ZZZ(y) = \varnothing$. In what follows, we may  assume
$\ZZZ(y)$ and $\ZZZ(\widetilde{\varphi}^{\ast}(x))$ are nonempty, which means $\OOO_{X, x}$ and $\OOO_{Y, y}$ are Dedekind domains.

Now suppose $\frakp \in \ZZZ(y)$. 
Then $g(\frakp) = 0$ for any $g \in y$ and hence $\frakp \cap T_y = \varnothing$. This implies
$\frakP = T_y^{-1}\frakp$ is a nonzero prime ideal of $T_y^{-1}A$. Let $\frakB$ be the ideal in $S_x^{-1}B$ generated by $\Phi_x^{\#}(\frakP)$. As $\Phi_x^{\#}$ is injective, $\frakB$ is nonzero, and by the assumption about stalk maps, $\frakB$ is a proper ideal. Therefore there is a nonzero prime ideal $\frakQ \subset S_x^{-1}B$ such that $\frakQ \mid \frakB$, and then
$\frakQ = S_x^{-1}\frakq$ for a nonzero prime $\frakq \subset B$ satisfying $\frakq \cap S_x = \varnothing$. This implies $\frakq \in \ZZZ(x)$. Then $\frakp' = \varphi^{-1}(\frakq)$ is a nonzero prime of $A$, and since (\ref{diagram2}) commutes, $\frakp \subset \frakp'$.
As $\frakp$ is maximal, $\frakp = \frakp' = \varphi^{-1}(\frakq)$, and since
$\ZZZ(\widetilde{\varphi}^\ast(x)) = \{\varphi^{-1}(\mathfrak{r}) : \mathfrak{r} \in \ZZZ(x)\}$, 
we have $\frakp \in \ZZZ(\widetilde{\varphi}^\ast(x))$.

Conversely, suppose $\frakp \in \ZZZ(\widetilde{\varphi}^\ast(x))$. Then $\frakp = \varphi^{-1}(\frakq)$ for some $\frakq \in \ZZZ(x)$, which implies $\frakq \cap S_x = \varnothing$. Hence
$\frakQ = S_x^{-1}\frakq$ is a nonzero prime ideal in $S_x^{-1}B$. Let $\frakP = (\Phi_x^{\#})^{-1}(\frakQ)$, which is a nonzero prime ideal in $T_y^{-1}A$. Then $\frakP = T_y^{-1}\frakp'$ for some nonzero prime ideal $\frakp' \subset A$ satisfying $\frakp' \cap T_y = \varnothing$. By the commutativity of (\ref{diagram2}), $\frakp \subset \frakp'$ and hence
$\frakp = \frakp'$. This implies $\frakp \cap T_y = \varnothing$, so $\frakp \in \ZZZ(y)$.
Therefore $\ZZZ(y) = \ZZZ(\widetilde{\varphi}^\ast(x))$, which shows $\Phi = \widetilde{\varphi}^\ast$ as continuous maps.

Finally, by the universal property of localization, there is a unique homomorphism across the bottom row making the diagram (\ref{diagram2}) commute. Since $y = \Phi(x) = \widetilde{\varphi}^\ast(x)$ and the stalk map $\widetilde{\varphi}_x^{\#}$ induced by 
$\widetilde{\varphi}^{\#}$ also makes the diagram commute, $\widetilde{\varphi}_x^{\#} = \Phi_x^{\#}$. Since this is true for all $x \in X$, the sheaf morphisms $\Phi^{\#}$ and $\widetilde{\varphi}^{\#}$ are equal.
\end{proof}

If $\mathbf{B}_0$ is the essential image of $\MMM_1$, the proof shows that a quasi-inverse of $\MMM_1 : \mathbf{A} \to \mathbf{B}_0$ is the functor
$\Gamma : \mathbf{B}_0 \to \mathbf{A}$ given on objects by $\Gamma(X, \OOO_X) = \OOO_X(X)$ and on morphisms by $$\Gamma((\Phi, \Phi^{\#}) : (X, \OOO_X) \to (Y, \OOO_Y)) = \Phi^{\#}(Y).$$

\section{Global spaces}

Having established the relationship between a Dedekind domain $A$ and the ``affine space" $\MMM_1(A)$, it is natural at this point to at least define and give some examples of a global analogue of $\MMM_1(A)$. Before making this definition, let us extend the notion of the local space to allow for dimension $0$. Let $k$ be a field and define $X = \MMM(k) = \MMM_1(k)$ to be the set of all totally multiplicative functions $I_k = \{(1)\} \to K$, so 
$\MMM_1(k) = \{f\}$, where $f((1)) = 1$. Give $X$ the only possible topology and define a sheaf of rings by $\OOO_X(X) = k$.

\begin{definition}
An \textit{affine arithmetic space} is a ringed space $(X, \OOO_X)$ isomorphic as a ringed space to $(\MMM_1(A), \OOO_{\MMM_1(A)})$ for some Dedekind domain or field $A$. A morphism between affine arithmetic spaces is a morphism in the category $\mathbf{B}$ defined in the previous section.
\end{definition}

Let $\mathbf{A}'$ be the category whose objects are Dedekind domains and fields and whose morphisms are injective, quasi-integral homomorphisms. Let $\mathbf{Aff}$ be the category of affine arithmetic spaces.

\begin{corollary}
The functor $\MMM_1 : \mathbf{A}' \to \mathbf{Aff}$ is an anti-equivalence of categories.
\end{corollary}

\begin{proof}
We just need to check that the map
$$
\MMM_1 : \Hom_{\mathbf{A}'}(A, B) \to \Hom_{\mathbf{Aff}}(\MMM_1(B), \MMM_1(A))
$$
is a bijection when $A$ or $B$ is a field. It is injective by the same argument as in the first paragraph of the proof of Theorem \ref{equivalence}. If $A$ is a field and $\varphi : A \to B$ is an injective,
quasi-integral homomorphism, then $\varphi^{-1}(\frakq) \neq (0)$ for all nonzero primes $\frakq \subset B$, which forces $B$ to also be a field. Hence, we may assume $B = k$ is a field. Let $X = \MMM_1(k) = \{f\}$ and $Y = \MMM_1(A)$. For any morphism $\varphi : A \to k$ in $\mathbf{A}'$, we have $\varphi^\ast(f)(\fraka) = f(k) = 1$ for all $\fraka \in I_A$, which means $\varphi^\ast(f) = \eta$ is the generic point of $Y$. 

Suppose $(\Phi, \Phi^{\#}) : X \to Y$ is a morphism in $\mathbf{Aff}$, so $\Phi^{\#}(Y) : A \to k$ is injective and quasi-integral. Then the stalk homomorphism $\Phi^{\#}_f : \OOO_{Y, \Phi(f)} \to \OOO_{X, f} = k$ is also injective. If $\frakm \subset \OOO_{Y, \Phi(f)}$ is a maximal ideal, then the ideal generated by $\Phi^{\#}_f(\frakm)$ in $k$ is proper by assumption, so it is zero. This gives $\frakm = (0)$, which means $\OOO_{Y, \Phi(f)} = \text{Frac}(A)$. Therefore $\Phi(f) = \eta$ and $\Phi = \varphi^\ast$, where $\varphi = \Phi^{\#}(Y)$.
 \end{proof}

\begin{definition}
An \textit{arithmetic space} is a ringed space $(X, \OOO_X)$ such that there is an open cover $\{U_i\}$ of $X$ with each $(U_i, \OOO_X|_{U_i})$ an affine arithmetic space. We call such an open cover an \textit{affine open cover} of $X$. A morphism of arithmetic spaces is a morphism of ringed spaces satisfying the same conditions as a morphism of affine arithmetic spaces.
\end{definition}

\begin{example}
Let $(X, \OOO_X)$ be an arithmetic space and $U \subset X$ an open subset. Then $(U, \OOO_X|_U)$ is an arithmetic space. To see this, let $x \in U$. Then $x \in V$ for some affine open set $V \subset X$, so $V \cong \MMM_1(A)$ for some $A$. We have seen that the open sets $\DDD_1(a) \cong \MMM_1(A_a)$
form a basis for the topology on $V$, where $a$ ranges over $A^\ast$, so there is such an open set with $x \in \DDD_1(a) \subset U \cap V$. This shows $U$ has an affine open cover. Additionally, this shows that the affine open subsets of an arithmetic space form a basis for its topology.
\end{example}

\begin{example}
Let $A$ and $B$ be objects of $\mathbf{A}'$ and set $X_1 = \MMM_1(A)$ and $X_2 = \MMM_1(B)$. Define
a topological space $X = X_1 \coprod X_2$ and a sheaf of rings on $X$ by
$$
\OOO_X(U) = \OOO_{X_1}(U \cap X_1) \times \OOO_{X_2}(U \cap X_2),
$$
so $\OOO_X|_{X_i} = \OOO_{X_i}$. Then $(X, \OOO_X)$ is an arithmetic space with affine open cover $\{X_1, X_2\}$. However, $X$ is not affine: if $X \cong \MMM_1(C)$ then 
$\OOO_X(X) \cong C$, but $\OOO_X(X) = A \times B$ is not a domain.
\end{example}

\begin{example}
Let $A$ be a PID which is not local and let $X = \MMM_1(A) \sm \{[e]\}$, where $[e] \in \MMM_1(A)$ is the closed point. Then $X \subset \MMM_1(A)$ is open, hence an arithmetic space. The space $X$ is not affine. To see this, note that
$$
X = \bigcup_{p \in A \text{ prime}} \DDD_1(p)
$$
implies
$$
\OOO_X(X) = \bigcap_{p \in A \text{ prime}} A_p = A
$$
since $A$ has at least two non-associate prime elements. It follows that if $X$ is affine then the inclusion $X \to \MMM_1(A)$ is an isomorphism, which is false since $[e]$ is not in the image. 

If $A$ is a discrete valuation ring then $X = \MMM_1(A) \sm \{[e]\} = \{\eta\}$, where $\eta$ is the generic point. Hence $\OOO_X(X) = A_p = \text{Frac}(A)$ for any prime element $p \in A$, which shows $X \cong \MMM_1(\text{Frac}(A))$ is affine.
\end{example}


\begin{thebibliography}{9}

\bibitem{Apostol}
Apostol, T. \textit{Introduction to Analytic Number Theory}. UTM, Springer, 1976.

\bibitem{Bourbaki1}
Bourbaki, N. \textit{Algebra} I. Springer, 1998.

\bibitem{Bourbaki2}
Bourbaki, N. \textit{Algebra} II. Springer, 2003.

\bibitem{CE}
Cashwell, E.D. and Everett, C.J. \textit{The ring of number-theoretic functions}. Pacific J. Math. \textbf{9} (1959), No. 4, 975-985.

\bibitem{Lang}
Lang, S. \textit{Algebra}. 3rd ed. GTM vol. 211, Springer, 2002.

\bibitem{Matsumura}
Matsumura, H. \textit{Commutative Ring Theory}. Cambridge University Press, 1989.

\bibitem{Nishimura}
Nishimura, H. \textit{On the unique factorization theorem for formal power series}. J. Math. Kyoto Univ. \textbf{7}-\textbf{2} (1967), 151-160.


\bibitem{SS2}
Schwab, E.D. and Silberberg, G. \textit{The valued ring of the arithmetical functions as a power series ring}. Arch. Math. \textbf{37} (2001), No. 1, 77-80.


\end{thebibliography}
\end{document}